\documentclass[12pt]{amsart}
\title{The \'etale-open topology and the stable fields conjecture}
\author{Will Johnson, Chieu-Minh Tran, Erik Walsberg, Jinhe Ye }
\email{willjohnson@fudan.edu.cn, mtran6@nd.edu, ewalsber@uci.edu, jinhe.ye@imj-prg.fr}

\usepackage{amsmath, amssymb, amsthm, enumitem, comment}    
\usepackage[mathscr]{euscript}
\usepackage[Symbol]{upgreek}
\usepackage{fullpage} 	
\usepackage{hyperref}
\usepackage{lineno}
\usepackage[all]{xy}
\usepackage{centernot}
\usepackage{xcolor}

\DeclareFontFamily{U}{fsy}{}
\DeclareFontShape{U}{fsy}{m}{n}{<->s*[.9]psyr}{}
\DeclareSymbolFont{der@m}{U}{fsy}{m}{n}
\DeclareMathSymbol{\der}{\mathord}{der@m}{182}

\DeclareFontFamily{U}{BOONDOX-calo}{\skewchar\font=45 }
\DeclareFontShape{U}{BOONDOX-calo}{m}{n}{
  <-> s*[1.05] BOONDOX-r-calo}{}
\DeclareFontShape{U}{BOONDOX-calo}{b}{n}{
  <-> s*[1.05] BOONDOX-b-calo}{}
\DeclareMathAlphabet{\mathcalboondox}{U}{BOONDOX-calo}{m}{n}
\SetMathAlphabet{\mathcalboondox}{bold}{U}{BOONDOX-calo}{b}{n}
\DeclareMathAlphabet{\mathbcalboondox}{U}{BOONDOX-calo}{b}{n}

\DeclareSymbolFont{imag@m}{OT1}{cmr}{m}{ui}
\DeclareMathSymbol{\imag}{\mathord}{imag@m}{105}

\DeclareMathOperator*{\forkindep}{\raise0.2ex\hbox{\ooalign{\hidewidth$\vert$\hidewidth\cr\raise-0.9ex\hbox{$\smile$}}}}

\newcommand{\Sa}[1]{\ensuremath{\mathscr{#1}}}

\newcommand{\reg}{\mathrm{reg}}
\newcommand{\sm}{\mathrm{sm}}
\newcommand{\Def}{\mathrm{Def}}
\newcommand{\kalg}{K^{\mathrm{alg}}}
\newcommand{\palg}{\overline{\Ff}_p}
\newcommand{\bk}{\mathbf{k}}
\newcommand{\Gal}{\operatorname{Gal}}
\newcommand{\Spec}{\operatorname{Spec}}

\newcommand{\Frac}{\operatorname{Frac}}

\newcommand{\val}{\operatorname{val}}

\newcommand{\Chara}{\operatorname{Char}}

\newcommand{\reslk}{\rtp}
\newcommand{\affk}{\mathrm{Aff}_K}
\newcommand{\affl}{\mathrm{Aff}_L}

\newcommand{\etp}{\mathrm{Ext}_{L/K}}
\newcommand{\rtp}{\mathrm{Res}_{L/K}}

\newtheorem*{claim-star}{Claim}
\newtheorem{theorem}{Theorem}[section] 
\newtheorem{lemma}[theorem]{Lemma}

\newtheorem{prop-def}[theorem]{Proposition-Definition}
\newtheorem{corollary}[theorem]{Corollary}
\newtheorem{fact}[theorem]{Fact}
\newtheorem{fact-eh}[theorem]{Fact(?)}

\newtheorem{proposition}[theorem]{Proposition}
\newtheorem{proposition-eh}[theorem]{Proposition(?)}
\newtheorem*{theorem-star}{Theorem}
\newtheorem*{conjecture-star}{Conjecture}
\newtheorem*{lemma-star}{Lemma}

\newtheorem*{thmA}{Theorem A}
\newtheorem*{thmB}{Theorem B}
\newtheorem*{thmC}{Theorem C}
\newtheorem*{thmD}{Theorem D}

\theoremstyle{definition}
\newtheorem{definition}[theorem]{Definition}

\newtheorem{remark}[theorem]{Remark}
\theoremstyle{remark}

\newcommand{\Aa}{\mathbb{A}}
\newcommand{\Ff}{\mathbb{F}}
\newcommand{\Gg}{\mathbb{G}}
\newcommand{\Qq}{\mathbb{Q}}
\newcommand{\Rr}{\mathbb{R}}

\newcommand{\Nn}{\mathbb{N}}
\newcommand{\Cc}{\mathbb{C}}
\newcommand{\Pp}{\mathbb{P}}

\newcommand{\cE}{\mathscr{E}}
\newcommand{\Var}{\mathrm{Var}}

\newcommand{\cT}{\mathscr{T}}
\newcommand{\cS}{\mathscr{S}}

\newcommand{\meno}{\medskip \noindent}

\newenvironment{claimproof}[1][\proofname]
               {
                 \proof[#1]
                 
               }
               {
                 \endproof
               }


\begin{document}
\maketitle
\begin{abstract}
For an arbitrary field $K$ and $K$-variety $V$, we introduce the \'etale-open topology on the set $V(K)$ of $K$-points of $V$.
This topology agrees with the Zariski topology, Euclidean topology, or valuation topology when $K$ is separably closed, real closed, or $p$-adically closed, respectively.
Topological properties of the \'etale-open topology corresponds to algebraic properties of $K$.
For example, the \'etale-open topology on $\mathbb{A}^1(K)$ is not discrete if and only if $K$ is large.
As an application, we show that a large stable field is separably closed.
\end{abstract}

\section{Introduction}
\noindent 
We introduce the \textbf{\'etale-open topology} on the set $V(K)$ of $K$-points of a variety $V$ over a field $K$, show that this generalizes the natural topology for many choices of $K$, and study the relationship between the properties of this topology and algebraic properties of $K$.
In particular, the \'etale-open topology is non-discrete if and only if $K$ is large.
Recall that $K$ is \textbf{large} if every smooth $K$-curve with a $K$-point has infinitely many $K$-points.
Separably closed fields, real closed fields, fields which admit non-trivial Henselian valuations such as $\Qq_p$ and $K((t))$, quotient fields of Henselian domains, pseudofinite fields, infinite algebraic extensions of finite fields, $\mathrm{PAC}$ fields, $p$-closed fields, and fields which satisfy a local-global principle (such as pseudo real closed and pseudo $p$-adically closed fields) are all large.
Finite fields, number fields, and function fields are not large. 
See \cite{Pop-little} for a survey of largeness and further examples. 
All known model-theoretically tame infinite fields seem to be large.

\medskip
We use the \'etale-open topology to resolve the model-theoretic \emph{stable fields conjecture} for large fields.
Specifically, we prove that large stable fields are separably closed.
The assumption of largeness might be necessary: Scanlon has recently raised the question of whether the (nonlarge) field $\mathbb{C}(t)$ is stable, giving suggestive evidence from arithmetic geometry \cite{Scanlon-CT}.



\meno
Throughout, $K$ is an infinite field, $X$, $Y$,  $V$, and $W$ range over $K$-varieties (i.e.\@ separated schemes of finite type over $K$), $K[V]$ is the coordinate ring of $V$, and $V(K)$ is the set of $K$-points of $V$.
The notion of an \'etale morphism was introduced by Grothendieck as the algebro-geometric counterpart of a local homeomorphism. If $V$ is smooth then $f: X \to V$ is {\bf \'etale} if and only if $X$ is smooth and  $f$ induces a tangent space isomorphism $T_p X \to T_{f(p)} V$ at every $p \in X$; see Section~\ref{section:basic etale} for the more technical definition when $V$ is not smooth.
A subset $U$ of $V(K)$ is an \textbf{\'etale image in} $V(K)$ if there is an \'etale morphism $X \to V$ such that $U$ is the image of the induced map $X(K) \to V(K)$.
Standard results about \'etale morphisms imply that the collection of \'etale images in $V(K)$ is closed under finite intersections and finite unions, and contains all Zariski-open subsets of $V(K)$; see Lemmas~\ref{intersection:lem2a} and \ref{lem:intersection} for details.
In particular, this collection is a basis for a topology  (in the classical sense) on $V(K)$, which we call the {\bf \'etale-open topology} on $V(K)$.

\begin{remark}
We make some comments concerning the relationship between the \'etale-open topology and the classical \'etale topology developed by the Grothendieck school.
Both are defined in terms of \'etale morphisms and are motivated by similar intuition.
However, we are not aware of a direct connection.
The \'etale topology is a Grothendieck topology (i.e., a site), whereas the \'etale-open topology is a topology in the classical sense.
The \'etale topology has well-behaved sheaf cohomology, and is often ``connected'' in spirit.
In contrast, the \'etale-open topology is usually totally disconnected (Theorem~\ref{thm:connected}).
There is not even a comparison morphism between the \'etale site and the \'etale-open site.
\end{remark}

Let $\cE_K = \{\cE_{V}\}$ be the family consisting of the \'etale-open topology on $V(K)$ for each $K$-variety $V$.
This family satisfies some natural compatibility conditions which we now describe.

\meno
Suppose $\cT = \{\cT_V\}$ is a family consisting of a topology on $V(K)$ for each $K$-variety $V$.  We refer to $\cT_V$ as the \textbf{$\cT$-topology} on $V(K)$.
\begin{definition}\label{sys-top:def}
The family $\cT$ is a {\bf system of topologies} if for any morphism  $f : V \to W$:
\begin{enumerate}
\item \label{st-1} the
induced map $V(K) \to W(K)$ is $\cT$-continuous.
\item \label{st-2} If $f$ is a (scheme-theoretic) open immersion, then the induced map $V(K) \to W(K)$ is a $\cT$-open embedding.
\item \label{st-3} If $f$ is a (scheme-theoretic) closed immersion, then the induced map $V(K) \to W(K)$ is a $\cT$-closed embedding.
\end{enumerate}
\end{definition}

Note that a system of topologies determines a functor from the category of $K$-varieties to the category of topological spaces.  More precisely, a system of topologies gives a lifting of the $K$-points functor (from varieties to sets)
along the forgetful functor from topological spaces to sets.
Two systems of topologies are already familiar: the \textbf{Zariski system of topologies} assigns to each variety $V$ the Zariski topology on $V(K)$.  Likewise, any field topology on $K$ determines a system of topologies over $K$; see Fact~\ref{fact:field-top} below.

\begin{thmA}
The family $\cE_K$ of \'etale-open topologies is a system of topologies. 
\end{thmA}

Theorem A is an easy consequence of standard results on \'etale morphisms; see Section~\ref{section:thmA}.

\medskip


When $K$ is separably closed, the \'etale-open topology agrees with the Zariski topology; see Proposition~\ref{prop:zar}.
When $K$ is real closed, the \'etale-open topology
is induced by the order topology; see Corollary~\ref{rcf-case}.
The \'etale-open topology
is induced by the Henselian valuation topology
when $K$ is a non-separably closed henselian valued field such as $\Qq_p$; see Corollary~\ref{hens-case}.
In particular, if $K$ is a local field other than $\Cc$, then the \'etale-open topology is induced by the usual field topology on $K$.


\meno
More generally, the \'etale-open topology is closely connected to t-Henselianity.
Recall that a \textbf{V-topology} on $K$ is a non-discrete field topology induced by an absolute value or valuation.
V-topologies can also be characterized intrinsically; see Definition~\ref{defn:V}.
A V-topology is \textbf{t-Henselian} if it satisfies a  topological analogue of Hensel's lemma; see Definition~\ref{defn:t-hensel}.
The usual topology on a local field, the valuation topology on a Henselian valued field, and the order topology on a real closed field are all t-Henselian.
If $K$ is not separably closed then $K$ admits at most one t-Henselian topology.
So we say that $K$ is t-Henselian if $K$ admits a t-Henselian field topology, and if $K$ is not separably closed then we refer to this canonical topology as ``the'' t-Henselian topology.

\begin{thmB}
\label{thm:intro-1}
If $K$ is t-Henselian and not separably closed then the \'etale-open topology over $K$ is induced by the t-Henselian topology.
If the \'etale-open topology over $K$ is induced by a V-topology $\uptau$ on $K$ then $\uptau$ (and hence $K$) is t-Henselian and $K$ is not separably closed.
\end{thmB}

Every $t$-Henselian field is large.
Hence, we can view the \'etale-open topology as a natural generalization of the t-Henselian topology to the broader class of large fields.
In general the \'etale-open topology $\cE_K$ is not induced by a field topology on $K$; see Section~\ref{section:pseudofinite}.

\medskip
We expect topological properties of the \'etale-open topology to correspond to field-theoretic properties of $K$.
The following theorem is a first step in this direction.

\begin{thmC} 
\ 
\begin{enumerate} 
\item $K$ is large if and only if the \'etale-open topology on $\Aa^1(K)$ is not discrete if and only if the \'etale-open topology on $V(K)$ is non-discrete whenever $V(K)$ is infinite.
\item $K$ is not separably closed if and only if the \'etale-open topology on $V(K)$ is Hausdorff for quasi-projective $V$.
\item The \'etale-open topology on $\Aa^1(K)$ is connected if and only if $K$ is separably closed or isomorphic to $\Rr$.
(More generally: the \'etale-open topology on $\Aa^1(K)$ is definably connected if and only if $K$ is separably closed or real closed.)  
\item The \'etale-open topology on $\Aa^1(K)$ is locally compact Hausdorff if and only if $K$ is a local field other than $\Cc$.
\end{enumerate}
\end{thmC}

We believe that the \'etale-open topology will be a useful tool in the model theory of large fields.
As evidence of this we offer Theorem D below, a special case of a famous conjecture.

\begin{thmD}
If $K$ is large and stable then $K$ is separably closed.
\end{thmD}

Here, ``stable'' is in the sense of model theory; see \cite{Poizat}.
We first proved Theorem D as a corollary of Theorems A and C, as presented in Section~\ref{sec: tablefieldconj2}.
Later we extracted a self-contained proof, in which the \'etale-open topology only appears implicitly.  
For the reader's convenience, we give this direct proof in Section~\ref{sec: tablefieldconj}.
%
%
Our proof of Theorem D in Section~\ref{sec: tablefieldconj} produces an explicit unstable formula in a non-separably closed virtually large field.

\meno
Theorem D is a partial result towards the conjecture that an infinite stable field is separably closed.
Any separably closed field is stable by work of Ershov~\cite{Ershov-solvable} and Wood~\cite{Wood1979}.
Macintyre~\cite{Macintyre-omegastable} showed that an infinite $\aleph_0$-stable field is algebraically closed.
About ten years later Cherlin and Shelah~\cite{Cherlin-Shelah-superstable} showed that an infinite superstable field is algebraically closed.
In $1999$ Scanlon showed that an infinite stable field is Artin-Schreier closed \cite{Kaplan2011}, and there has been little progress in the past twenty years.

\meno
There is another notable model-theoretic conjecture which is open in general but holds for large fields: the Podewski conjecture that a minimal field is algebraically closed.
Koenigsmann showed that the Podewski conjecture holds for large fields.
This also follows immediately from the fact that the \'etale-open topology on a large non-separably closed field is non-discrete and Hausdorff; see Section~\ref{section:podewski}.

\meno
We now summarize the content of this paper.
In Section~\ref{sec: tablefieldconj} we prove Theorem D.
In Section~\ref{section:general-facts} we prove some general facts about systems of topologies.
In particular we show that if $L$ is an extension of $K$ then any system over $L$ ``restricts'' to a system over $K$ and if $L$ is a finite extension of $K$ then any system over $K$ ``extends'' to a system over $L$.
We will make extensive use of these operations.
In Section~\ref{section:thmA} we prove Theorem A and some useful facts about extension and restriction of the \'etale-open system.
In Section~\ref{sec:example} we prove the first claim of Theorem B.
In that section we also show that if $<$ is a field order on $K$ then $\cE_K$ refines the system of topologies induced by $<$ and if $v$ is a non-trivial valuation on $K$ with non-separably closed Henselization then $\cE_K$ refines the system of topologies induced by $v$.
In particular if $v$ is a non-trivial valuation with either non-algebraically closed residue field or non-divisible value group then $\cE_K$ refines the system induced by $v$.
In Section~\ref{sec:topo-field} we prove Theorem C and the second claim of Theorem B.
Finally, in Section~\ref{section:pseudofinite} we give some examples of $K$ such that $\cE_K$ is not a field topology.
We use the Hasse-Weil bounds to show that the $\cE_K$-topology on $K$ is not a field topology when $K$ is an infinite non-quadratically closed algebraic extension of a finite field.
We also give a model-theoretic proof that the \'etale-open topology on a pseudofinite field of odd characteristic is not a field topology.

\subsection{Acknowledgements}
We thank Ehud Hrushovski for useful conversations and thank the anonymous referee for many useful comments on the structure of the paper.
The first author acknowledges support by the National Science Foundation under Award No.\@ DMS-1803120. The fourth author was partially supported by NSF grant DMS-1800806.  Any opinions, findings, and conclusions or recommendations expressed in this material are those of the authors and do not necessarily reflect the views of the National Science Foundation.

\section*{Notations and conventions}
\noindent
Throughout, $n$ is a natural number, $k,l,m$ are integers, and $K$ is an infinite field.
We let $\Chara(K)$ denote the characteristic of $K$.
A \textbf{variety} over $K$, or $K$-variety, is a separated scheme of finite type over $K$, not assumed to be reduced or irreducible.
We let $\Var_K$ be the category of $K$-varieties.
We let $V(K)$ denote the set of $K$-points of a $K$-variety $V$.
We refer to $V \mapsto V(K)$ as the \textbf{$K$-points functor}.  Each $K$-point determines a (scheme-theoretic) point of $V$.  The \textbf{Zariski topology} on $V(K)$ is the topology making $V(K) \to V$ be a topological embedding.

\meno
An \textbf{open (closed) immersion} is an open (closed) immersion of schemes.
An \textbf{open (closed) embedding} is an open (closed) embedding of topological spaces.
If $\uptau$ is a topology on a set $X$ we will sometimes write $(X,\uptau)$ to denote the topological space.  
A topological space $X$ is \textbf{totally separated} if for any $a, b \in X$ there is a clopen $U \subseteq X$ such that $a \in U$ and $b \notin U$.

\meno
We let $\Aa^n$ and $\Gg_m$ denote the varieties $\Spec K[x_1,\ldots,x_n]$ and $\Spec K[y,y^{-1}]$, i.e., affine space and the multiplicative group over $K$.  
Recall that $\Aa^n(K)$ is $K^n$ and $\Gg_m(K)$ is $K^\times$.
Given a $K$-variety $V$ and an extension $L/K$ we let $V_L = V \times_{\Spec K} \Spec L$ be the base change of $V$ and if $f : V \to W$ is a morphism of $K$-varieties then $f_L : V_L \to W_L$ is the base change of $f$.
(The base change $V_L$ can fail to be reduced when $V$ is reduced and $L/K$ is purely inseparable; this is why we do not assume that $K$-varieties are reduced.)

\section{Preliminaries}

\subsection{Algebro-geometric preliminaries} \label{section:basic etale}

\meno
We will need some basic facts on smooth and regular points.
We let $V_{\mathrm{sm}}, V_\mathrm{reg}$ be the smooth and regular loci of a $K$-variety $V$, respectively.

\begin{fact}
\label{fact:smooth points}
Suppose that $V$ is a $K$-variety and $p \in V(K)$.
\begin{enumerate}
\item $V_\mathrm{sm}$ and $V_\mathrm{reg}$ are both open subvarieties of $V$.
\item $p \in V_\mathrm{sm}$ if and only if $p \in V_\mathrm{reg}$.
\item $V_\reg$ is nonempty.
\end{enumerate}
\end{fact}

\begin{proof}
(1) is \cite[Lemma~056V, Proposition~07QW]{stacks-project}, (2) is~\cite[Lemma~00TV]{stacks-project}, and (3) follows as the generic point of an irreducible component of $V$ is regular.
\end{proof}

A {\bf standard \'etale morphism} is a morphism $\pi : X \to V$ where $V$ is a $K$-variety, $X$ is the subvariety of $V \times \Aa^1$ given by $f = 0$, $g \ne 0$ for $f,g \in (K[V])[y]$ such that $f$ is monic, $\der f/\der y \ne 0$ on $X$, and $\pi$ is the restriction of the projection $V \times \Aa^1 \to V$. A $K$-variety morphism is {\bf \'etale} if it is locally a standard \'etale morphism up to isomorphism.


\begin{fact}
\label{fact:etale}
\quad
\begin{enumerate}
\item  Open immersions are \'etale.
\item \'Etale morphisms are closed under composition.
\item \'Etale morphisms are closed under base change.
\item Open immersions and closed immersions are closed under base change.
\item If $f : V \to W$ is an \'etale morphism of $K$-varieties then the image of $f$ is an open subvariety of $V$.
\end{enumerate}
\end{fact}

\begin{proof}
(1), (2), (3), (5) are \cite[Proposition~17.1.3]{EGA-IV-4} and (4) is \cite[4.3.2]{EGA-I}.
Alternatively (1),(2),(3),(4),(5) are ~\cite[Lemma 02GP, 02GN, 02G0, 01JY,  03WT]{stacks-project} respectively.
\end{proof}

Fact~\ref{fact:ega} is a special case of \cite[Proposition~18.1.1]{EGA-IV-4}.

\begin{fact}
\label{fact:ega}
Suppose that $W$ is a $K$-variety, $V$ is a closed subvariety of $W$, $X$ is a $K$-variety, $g : X \to V$ is an \'etale morphism, and $p$ is a (scheme-theoretic) point of $X$.
Then there is an open subvariety $O$ of $X$ with $p \in O$, a $K$-variety $Y$, and an \'etale morphism $h: Y \to W$ such that $O$ is isomorphic as a $V$-scheme to the fiber product $Y \times_W V $.
\end{fact}

We will need some basic results on Weil restriction of scalars.
We refer to \cite[A.5]{pseudo-reductive-groups} for a general account of Weil restriction.
Let $\affk$, $\affl$ be the categories of affine $K$-, $L$-varieties, respectively.
The {\bf Weil restriction} functor $\reslk : \affl \to \affk$ is right adjoint to the base change functor $\affk \to \affl, V \mapsto V_L$. (An arbitrary non-affine variety need not have a Weil restriction.)
We focus our attention on affine varieties  as a system of topologies is determined by its restriction to affine varieties.

\meno
We give an explicit definition of the Weil restriction $\reslk(V)$ of an affine $L$-variety $V$ for the benefit of the reader unfamiliar with this notion.
Let $V$ be $\Spec L[x_1,\ldots,x_n]/(f_1,\ldots,f_k)$.
Let $y_{ij}$ be a variable for each $i \in \{1,\ldots,n\}$ and $j \in \{1,\ldots,m\}$.
For each $i$ we substitute $y_{i1}e_1 + \cdots + y_{im} e_m$ for $x_i$, i.e for each $l \in \{1,\ldots,k\}$, $r \in \{1,\ldots,m\}$ let $g_{lr} \in K[y_{ij}]$ be the unique polynomial so that
\[ f_l\left( \sum_{j = 1}^{m} y_{1j} e_j,\ldots, \sum_{j = 1}^{m} y_{nj} e_j \right) = g_{l1}e_1 + \cdots + g_{lm}e_m .\]
Then $\reslk(V)$ is $\Spec K[y_{ij}]/(g_{lr})$.

\meno
Suppose that $V$ is an affine $K$-variety.
As base change and Weil restriction are adjoint functors there is a natural morphism $V \to \reslk(V_L)$.
We give an explicit description of this morphism under the assumption that $e_1 = 1$.
Suppose $V = \Spec K[x_1,\ldots,x_n]/(f_1,\ldots,f_k)$.
Then $V_L$ is $\Spec L[x_1,\ldots,x_n]/(f_1,\ldots,f_k)$.
Let $y_{ij}$ and $g_{lr}$ be defined as above, so $\reslk(V_L)$ is $\Spec K[y_{ij}]/(g_{lr})$.
Let $\varphi$ be the $K$-algebra morphism $K[y_{ij}]/(g_{lr})\to K[x_1,..,x_n]/(f_1,...,f_k)$ given by declaring $\varphi(y_{ij}) = x_i$ when $j=1$ and $\varphi(y_{ij}) = 0$ otherwise.
The canonical morphism $V \to \reslk(V_L)$ is the morphism corresponding to $\varphi$.
Fact~\ref{fact:can-is-closed} below follows from the observation that $\varphi$ is surjective.

\begin{fact}
\label{fact:can-is-closed}
Suppose that $L$ is a finite extension of $K$ and $V$ is an affine $K$-variety.
The canonical morphism $V \to \reslk(V_L)$ is a closed immersion.
\end{fact}

We recall some standard facts about Weil restriction of affine varieties (these facts hold whenever the Weil restriction exists).

\begin{fact} 
\label{fact:weil-restriction}
Suppose $L$ is a finite extension of $K$ and $[L:K] = m$.

\begin{enumerate}[leftmargin=*]
\item $\reslk(\Aa^n_L)$ is isomorphic to $\Aa^{mn}$.
\item If $V$ is an affine $L$-variety, then there is a canonical bijection $(\reslk (V) ) (K) \to V(L)$.
\item If $V,W$ are affine $L$-varieties and $f : V \to W$ is a morphism then the Weil restriction $(\reslk f) : (\reslk (V) )(K) \to (\reslk (W) )(K)$ agrees with $f : V(L) \to W(L)$ 
under the canonical bijection.
\end{enumerate}
\end{fact}

\begin{fact}
\label{fact:restrict-0}
Suppose that $f : V \to W$ is a morphism of affine $L$-varieties.
If $f$ is an open immersion then $\reslk(f)$ is an open immersion.
If $f$ is a closed immersion then $\reslk(f)$ is a closed immersion.
\end{fact}
 
See \cite[Proposition~A.5.2(4), Proposition~A.5.5]{pseudo-reductive-groups} for a proof of Fact~\ref{fact:restrict-0}.

\begin{fact}
\label{fact:restriction-etale}
If $f$ is an \'etale morphism between affine $L$-varieties then $\reslk(f)$ is \'etale.
\end{fact}

See~\cite[Proposition~A.5.2(4)]{pseudo-reductive-groups} for a proof of Fact~\ref{fact:restriction-etale}.

\subsection{Model-theoretic preliminaries} \label{sec:model-theoretic preliminaries}
Suppose $K$ is a definable field in some structure, and $X$ is a definable subset of $K$.
Recall that $X$ is \textbf{additively generic} if there is a finite $A \subseteq K$ such that $X + A = K$, $X$ is \textbf{multiplicatively generic} if there is finite $A \subseteq K^\times$ such that $A (X \cap K^\times) = K^\times$, and a partial unary type $p$ \textcolor{olive}{in} $K$ is \textbf{additively (multiplicatively) generic} if every formula in $p$ defines an additively (multiplicatively) generic set.  The following is well-known.

\begin{fact}[Theorem 5.10 \cite{Poizat}]
\label{fact:poizat}
Suppose that $K$ is stable. 
Then there is a unique complete additive generic type $p^+$ and a unique complete multiplicative generic type $p^\times$, and $p^+ = p^\times$.  Moreover, a definable $X \subseteq K$ is generic if and only if $p^+$ concentrates on $X$.
\end{fact}

We now prove a local version of Fact~\ref{fact:poizat}. The proof is a straightforward localization of the usual proof of Fact~\ref{fact:poizat}.
We first recall some facts about local generics in definable groups.
\textbf{Throughout, $G$ is a definable group in a structure $\Sa M$.}

\meno
Suppose that $\delta(x,y)$ is a formula such that $\alpha \in G$ whenever $\Sa M \models \delta(\alpha,b)$ for some $b$.
Then $\delta$ is \textbf{invariant} if for any $b$ and $\alpha \in G$ there is $b^*$ such that $\alpha\delta(G,b) = \delta(G,b^*)$.
Let $\Def_\delta(G)$ be the boolean algebra generated by instances of $\delta$ and $S_\delta(G)$ be the set of complete $\delta$-types.
If $\delta$ is invariant and $X \in \Def_\delta(G)$ then $\alpha X \in \Def_\delta(G)$ for any $\alpha \in G$.
A subset $Y$ of $G$ is \textbf{generic} if there are $\alpha_1,\ldots,\alpha_n \in G$ such that $\alpha_1 Y \cup \cdots \cup \alpha_n Y = G$, and $p \in S_\delta(G)$ is generic if it only contains generic sets.
Given $p \in S_\delta(G)$ and $\alpha \in G$ we define $\alpha p = \{ \alpha X : X \in p\}$.


\begin{fact} [Thm 2.3 \cite{CPT}]
\label{fact:cpt}
Suppose that $\delta(x,y)$ is stable and invariant.
Then there is a finite index subgroup $G^0_\delta$ of $G$ such that $G^0_\delta$ is in $\Def_\delta(G)$ and $G^0_\delta \subseteq H$ for any finite index subgroup $H \in \Def_\delta(G)$ of $G$.
Furthermore
\begin{enumerate}
\item Every left coset of $G^0_\delta$ contains a unique generic $\delta$-type.
\item If $\alpha \in G$ and $X \in \Def_\delta(G)$ then exactly one of $\alpha G^0_\delta \cap X$ or $\alpha G^0_\delta \setminus X$ is generic.
\end{enumerate}
\end{fact}

This is the local version of a well-known result on stable groups.
We say that $G$ is \textbf{$\delta$-connected} if $G^0_\delta = G$.

\begin{corollary}
\label{cor:cpt}
Suppose that $\delta(x,y)$ is stable and invariant, and $G$ is $\delta$-connected.
Then there is a unique generic type $p \in S_\delta(G)$ and any $X \in \Def_\delta(G)$ is generic if and only if $p$ concentrates on $X$.
If $X \in \Def_\delta(G)$ then exactly one of $X$ or $G \setminus X$ is generic.
\end{corollary}

We now suppose that $K$ is a definable field in a structure $\Sa M$.
Let $\varphi(x,y)$ be an formula such that $\alpha \in K$ whenever $\Sa M \models \varphi(\alpha,b)$ for some $b$.
We say that $\varphi$ is \textbf{affine invariant} if for any $b$, $\alpha \in K^\times, \beta \in K$ there is $b^*$ such that $\alpha\varphi(K,b) + \beta = \varphi(K,b^*)$.
Suppose that $\varphi$ is affine invariant.
We say that $p \in S_\varphi(K)$ is an \textbf{additive generic} if it is generic for $(K,+)$ and is a \textbf{multiplicative generic} if is generic for $K^\times$.  

\begin{proposition}
\label{prop:local-generics}
Suppose that $K$ is an $\Sa M$-definable field and $\varphi(x,y)$ is stable and affine invariant.
Then there is a unique additive generic $p_+ \in S_\varphi(K)$, a unique multiplicative generic $p_\times \in S_\varphi(K)$, and $p_+ = p_\times$.
\end{proposition}

\begin{proof}
We first show that there is a unique additive generic.
By Fact~\ref{fact:cpt} it suffices to show that $(K,+)$ is $\varphi$-connected.
Let $H$ be the connected component $(K,+)^0_\varphi$.  
For any $\beta \in K^\times$, the group $\beta^{-1}H$ is in $\Def_\varphi(K)$ and has finite index, and so $\beta^{-1}H \supseteq H$.  
Equivalently, $\beta H \subseteq H$. 
Then $H$ is an ideal, so $H = K$.

\meno
Let $p_+$ be the additive generic and $p_\times$ be any multiplicative generic.
We show that $p_\times = p_+$.  
Suppose otherwise.
Fix $X \in \Def_\varphi(G)$ such that $p_\times$ concentrates on $X$, and $p_+$ does not.  
Then $X$ is multiplicatively generic but not additively generic.
Hence there are $\alpha_1, \ldots, \alpha_n \in K^\times$ with $K^\times \subseteq \bigcup_{i = 1}^n \alpha_i X$.  
The map $x \mapsto \alpha_ix$ is an automorphism of $(K,+)$, so each $\alpha_i X$ is not additively generic.
Then $p_+$ does not concentrate on $\bigcup_{i = 1}^n \alpha_i X \supseteq K^\times$, which is absurd.
\end{proof}



\section{Direct proof of Theorem~D}
\label{sec: tablefieldconj}

Theorem D is an easy consequence of Theorems A and C---see Section~\ref{sec: tablefieldconj2} for a half-page proof.  Here, we give a self-contained proof of Theorem D, not using Theorems A or C, but extracted from our later proof in Section~\ref{sec: tablefieldconj2}.
In fact, we will prove the following strengthening of Theorem D:

\begin{theorem}\label{d-improved}
If $K$ is virtually large and every existential formula is stable, then $K$ is separably closed.
\end{theorem}
Here, we say that $K$ is \textbf{virtually large} if some finite extension of $K$ is large.
Srinivasan~\cite{Srinivasan} constructs a virtually large field which is not large.

\medskip
We will need local stability theory (Proposition~\ref{prop:local-generics}) to produce an explicit unstable formula; we do not need this to prove Theorem D.





\medskip
We assume the definitions from Section~\ref{sec:model-theoretic preliminaries}.  Suppose $\varphi$ is stable and affine invariant.  We say that a boolean combination of instances of $\varphi$ is \textbf{$\varphi$-generic} if the following equivalent conditions hold:
\begin{itemize}
    \item It is additively generic.
    \item It is multiplicatively generic.
    \item It contains $p^+_\varphi = p^\times_\varphi$.
\end{itemize}

\meno
We now gather some field-theoretic lemmas.

\begin{fact}
\label{fact:large-fin}
A finite extension of a large field is large.
\end{fact}

See \cite[Proposition~2.7]{Pop-little} for a proof.

\meno
Let $V$ be a variety over $K$.  Recall from the introduction that an \textbf{\'etale image} in $V(K)$ is a set of the form $f(X(K))$ where $f : X \to V$ is an \'etale morphism of $K$-varieties.
\begin{lemma}
\label{intersection:lem2a}
Let $V$ be a $K$-variety.  The collection of \'etale images in $V(K)$ is closed under finite unions and finite intersections.
\end{lemma}
\begin{proof}
  Suppose $U_0, U_1$ are \'etale images in $V(K)$.  For $i \in \{0,1\}$, let $f_i : X_i \to V$ be an \'etale morphism of $K$-varieties such that $f_i(X_i(K)) = U_i$.  Let $X_\cup$ be the disjoint union of $X_0$ and $X_1$ and $f_\cup : X_\cup \to V$ be the morphism produced from $f_0,f_1$.
  Then $f_\cup$ is \'etale and the image of the induced map $X_\cup \to V$ is $U_0 \cup U_1$.  Therefore $U_0 \cup U_1$ is an \'etale image.
  
  \meno
  Let $Y$ be the fiber product $X_0 \times_{V} X_1$, and $g : Y \to V$ be the natural map.  By Fact~\ref{fact:etale}, $g$ is \'etale.
Then
  \begin{equation*}
      \xymatrix{ Y(K) \ar[r] \ar[d] & X_0(K) \ar[d]^{f_0} \\ X_1(K) \ar[r]_{f_1} & V(K),}
  \end{equation*}
is a pull-back square, so $g(Y(K)) = f_0(X_0(K)) \cap f_1(X_1(K)) = U_0 \cap U_1$.
Therefore $U_0 \cap U_1$ is an \'etale image.
\end{proof}
\begin{lemma}
\label{intersection:lem2b}
Let $g : V \to W$ be an isomorphism of $K$-varieties.  If $U$ is an \'etale image in $V(K)$, then $g(U)$ is an \'etale image in $W(K)$.
\end{lemma}
\begin{proof}
Let $f : X \to V$ be an \'etale morphism of $K$-varieties such that $U = f(X(K))$.
Then $g \circ f : X \to W$ is \'etale, and $(g \circ f)(X(K)) =  g(U)$, so $g(U)$ is an \'etale image.
\end{proof}
In particular, if $U$ is an \'etale image in $K = \Aa^1(K)$, and $f : K \to K$ is an affine transformation $f(x) = ax + b$, then the direct image $f(U)$ is an \'etale image in $K$.





\begin{lemma}
\label{large:lem}
If $K$ is a large field, then every non-empty \'etale image in $K$ is infinite.
\end{lemma}

\begin{proof}
Let $S$ be a non-empty \'etale image in $K$.
There is a $K$-variety $X$ with $X(K) \ne \emptyset$ and an \'etale morphism $f : X \to \Aa^1$ with $S = f(X(K))$.
The variety $X$ is smooth of dimension 1, hence a curve.
By largeness, $X(K)$ is infinite.
As $f$ is finite-to-one, $S$ is infinite.
\end{proof}

Lemma~\ref{lemma:galois} below is essentially used in the proof of Macintyre's theorem on $\aleph_0$-stable fields.
We provide a proof for the sake of completeness.

\begin{lemma}
\label{lemma:galois}
Suppose that $K$ is not separably closed.
Then there are finite field extensions $L/K$ and $L'/L$ such that either
\begin{enumerate}
\item $L'$ is an Artin-Schreier extension of $L$, or
\item there is a prime $p \neq \mathrm{Char}(K)$ such that $L$ contains a primitive $p$th root of unity and $L' = L(\alpha)$ for some $\alpha \in L'$ such that $\alpha^p \in L$.
\end{enumerate}
\end{lemma}

\begin{proof}
As $K$ is not separably closed, $K$ has a non-trivial finite Galois extension.
Take $p > 1$ minimal such that there are finite extensions $K \subseteq L \subseteq L'$ with $L'/L$ Galois and $p = [L ' : L]$.
Take a prime $q$ dividing $p = |\Gal(L'/L)|$, and a subgroup $H$ of order $q$. 
Let $L''$ be the fixed field of $H$. Then $L'/L''$ is Galois and $[L' : L''] = q$.
By minimality, $p = q$, and $p$ is prime.

\meno  
If $p = \Chara(K)$ then $L'/L$ is an Artin-Schreier extension.
Suppose that $q \neq \Chara(K)$.
The extension of $L$ by a primitive $p$th rooth of unity is a Galois extension of degree $\leq p - 1$.  
So $L$ contains a primitive $p$th root of unity by minimality of $p$.
So $L'/L$ is a Kummer extension hence $L' = L(\alpha)$ for some $\alpha \in L'$ such that $\alpha^p = L$.
\end{proof}

The following is a famous theorem of Artin and Schreier.

\begin{fact}
\label{fact:artin}
If some finite extension of $K$ is separably closed then $K$ is either separably closed or real closed.
\end{fact}

\noindent
We now prove Theorem~\ref{d-improved}.

\begin{proof}
Suppose that $K$ is virtually large and not separably closed.
Let $L'/K$ be a large extension of minimal degree.
If $L'$ is separably closed then by Fact~\ref{fact:artin} $K$ is real closed, hence large, which contradicts minimality.
So we may suppose that $L'$ is not separably closed.
Applying Lemma~\ref{lemma:galois} we obtain a finite extension $L$ of $L'$ such that either
\begin{enumerate}
\item the $p$th power map $L^\times \to L^\times$ is not surjective for some prime $p \ne \Chara(K)$, or
\item the Artin-Schreier map $L \to L$ is not surjective.
\end{enumerate}
Note that $L$ is large by Fact~\ref{fact:large-fin}.
Fix a $K$-basis $\beta_1,\ldots,\beta_m$ of $L$.
We identify $L$ with $K^m$ by identifying every $(a_1,\ldots,a_m) \in K^m$ with $a_1 \beta_1 + \cdots + a_m\beta_m$.
Let $+_L,\times_L : K^m \times K^m \to K^m$ be addition and multiplication in $L$.
Note that $+_L, \times_L$ are both existentially definable in $K$.
In case $(1)$ fix such a $p$ and let $P$ be the image of the $p$th power map $L^\times \to L^\times$.
In case $(2)$ let $P$ be the image of the Artin-Schreier map $L \to L$.
Note that $P$ is existentially definable in $K$.
Let $\varphi(x,y_1,y_2)$ be the existential formula $[y_1 \in L^\times] \land [x \in (y_1 \times_L P) +_L y_2]$.
Note that $\varphi$ is affine invariant.
We show that $\varphi$ is unstable.

\meno
Let $\Sa B_L$ denote the set of \'etale images in $\Aa^1(L) = L$.  We first make a number of observations concerning $\Sa B_L$ and $\varphi$.  By Lemmas~\ref{intersection:lem2a}, \ref{intersection:lem2b}, $\Sa B_L$ is closed under finite intersections and affine transformations.  The $p$th power map for $p$ coprime to $\Chara(K)$ and the Artin-Schreier map are both \'etale, so in either case $P$ is in $\Sa B_L$.
Then every instance of $\varphi$ is in $\Sa B_L$, by affine symmetry.
In case $(1)$ $P$ is a non-trivial subgroup of $(L,+)$, in case $(2)$ $P$ is a non-trivial subgroup of $L^\times$, and in either case every coset of $P$ is an instance of $\varphi$.
So we may fix $P^* \subseteq K$ such that $P^*$ is defined by an instance of $\varphi$ and $P \cap P^* = \emptyset$.


\meno
Suppose $\varphi$ is stable.
By Proposition~\ref{prop:local-generics} either $P$ or $P^*$ is not $\varphi$-generic.
Suppose that $P$ is not $\varphi$-generic; the other case is similar.
Take $c \in P$ and let $P^{**} = P - c$.  Then $0 \in P^{**} \in \Sa B_L$, and $P^{**}$ is not $\varphi$-generic.  Thus, $L^\times \setminus P^{**}$ is $\varphi$-generic in $L^\times$ by Proposition~\ref{prop:local-generics}.
So there are $a_1,\ldots,a_n \in L^\times$ such that $\bigcup_{i = 1}^{n} a_i(L^\times \setminus P^{**}) = L^\times$.
Thus $\bigcap_{i = 1}^{n} a_i P^{**} = \{0\}$.
By Lemma~\ref{intersection:lem2a} $\{0\}$ is in $\Sa B_L$.
By Lemma~\ref{large:lem} this contradicts largeness of $L$.
\end{proof}

The unstable formula $\varphi$ produced in the proof of Theorem D is existential.
Corollary~\ref{cor:existential} follows by this and the fact that separably closed fields are stable.

\begin{corollary}
\label{cor:existential}
$K$ is stable when $K$ is virtually large and every existential formula is stable.
\end{corollary}

Corollary~\ref{cor:existential} is sharp in that any quantifier free formula in any field is stable.

\section{General facts on systems of topologies}
\label{section:general-facts}
It will be useful to know some general facts about systems of topologies (Definition~\ref{sys-top:def}).
Suppose $\cT$ and $\cT^*$ are both systems of topologies over $K$.
Then $\cT$ \textbf{refines} $\cT^*$ (and $\cT^*$ is \textbf{coarser} then $\cT$) if the $\cT$-topology on $V(K)$ refines the $\cT^*$-topology for any $K$-variety $V$.


\subsection{Reduction to affine varieties}

As every variety is locally affine, any system of topologies over $K$ will be determined by its restriction to affine $K$-varieties.
This is a consequence of the following trivial but useful remark.

\begin{remark} \label{rem:Sytemdeterminedbylocal}
Suppose $\cT$ is a system of topologies over $K$, $V$ is a $K$-variety, $(V_i)_{i \in I}$ is a collection of open subvarieties of $V$ covering $V$, and $A$ is a subset of $V(K)$.
Then $A \subseteq V(K)$ is $\cT$-open if and only if $A \cap V_i(K)$ is a $\cT$-open subset of $V_i(K)$ for all $i \in I$.
\end{remark}

Lemma~\ref{lem:determine} shows that it suffices to consider restrictions to affine spaces:

\begin{lemma}
\label{lem:determine}
If the $\cT'$-topology on $\Aa^n(K)$ refines the $\cT$-topology  on $\Aa^n(K)$ for every $n$ then $\cT'$ refines $\cT$.
If the $\cT'$-topology on $\Aa^n(K)$ agrees with the $\cT$-topology on $\Aa^n(K)$ for every $n$ then $\cT$ agrees with $\cT'$.
\end{lemma}

\begin{proof}
The second claim is immediate from the first.
Suppose that the $\cT'$-topology on $\Aa^n(K)$ refines the $\cT$-topology for every $n$.
As closed immersions induce $\cT$-closed embeddings,  the $\cT'$-topology on $V(K)$ refines the $\cT$-topology for every affine variety $V(K)$.
The statement for  general $V$ follows from Remark~\ref{rem:Sytemdeterminedbylocal} and the fact that every variety is locally affine.
\end{proof}

An \textbf{affine system of topologies} $\cS$ over $K$ is a choice of a topology $\cS_{V}$ on $V(K)$ for each affine $K$-variety $V$ such that the conditions of Definition~\ref{sys-top:def} hold for morphisms between affine varieties.
Every system of topologies over $K$ restricts to an affine system of topologies over $K$.
Lemma~\ref{lem:affine} gives us the converse.

\begin{lemma}\label{lem:affine}
Suppose that $\cS$ is an affine system of topologies over $K$.
Then there is a unique system of topologies $\cT$ over $K$ such that the $\cS$-topology on $V(K)$ agrees with the $\cT$-topology on $V(K)$ for any affine $K$-variety $V$.

\end{lemma}


\begin{proof}
Let $V$ be a $K$-variety.
Suppose $(V_i)_{i \in I}$ is a cover of $V$ by affine open subvarieties.
Let $\cT_V$ be the topology on $V(K)$ such that $U\subseteq V(K)$ is $\cT_V$-open if and only if $U\cap V_i(K)$ is $\cS_{V_i}$-open for all $i$.
We show that $\cT_V$ does not depend on choice of $(V_i)_{i \in I}$.
Suppose $(V'_j)_{j \in J}$ is another cover of $V$ by affine open subvarieties and likewise define $\cT'_V$ on $V(K)$.
We show that $\cT_V$ agrees with $\cT'_V$.
For each pair $V_i,V'_j$, take $O_{i,j}$ to be $V_i \cap V'_j$.
Then $O_{i,j}$ is an affine open subvariety of $V$ as $V$ is separated.
For any $U\subseteq V(K)$, $U\cap V_i(K)$ is in $\cS_{V_i}$ for all $i$ if and only if $U\cap O_{i,j}(K)$ is in $\cS_{O_{i,j}}$  for all $i,j$ if and only if $U\cap V'_j(K)$ is in $\cS_{V'_j}$ for all $j$. 

\meno
Let $\cT$ be the collection ${\cT_V}$.
It remains to show that  $\cT$ is a system of topologies.
We verify condition (\ref{st-2}) of Definition~\ref{sys-top:def} and leave the rest to the reader.
Let $f:V\to W$ be a open immersion of $K$-varieties.
Note that the induced map $V(K)\to W(K)$ is injective so it suffices to show that $V(K)\to W(K)$ is a $\cT$-open map. 
Suppose that $U$ is a $\cT$-open subset of $V(K)$; we show that $f(U)$ is $\cT$-open.
Let $(W_i)_{i \in I}$ be a cover of $W$ by affine open subvarieties.
It suffices to show that each $f(U)\cap W_i(K)$ is $\cT$-open in $W_i(K)$.
For each $W_i$ let $(V_{ij})_{i \in J_i}$ be a cover of $f^{-1}(W_i)$ by affine open subvarieties.
Then \[f(U)\cap W_{i}(K)=\bigcup_{j} f(U\cap V_{i,j}(K)).\]
Note that $f|_{V_{i,j}}:V_{i,j} \to W_{i}$ is $\cT$-open as $\Sa S$ is an affine system.
So the right hand side is a union of open sets, hence open.
\end{proof}

\subsection{The Zariski and discrete systems}
\label{section:zariski-discrete}
The {\bf Zariski system of topologies} is the system of topologies over $K$ that assigns the Zariski topology on $V(K)$ to each $K$-variety $V$.

\begin{lemma}
\label{lem:refine-Zariski}
Any system of topologies on $K$ refines the Zariski system of topologies.
\end{lemma}

\begin{proof}
Suppose that $\cT$ is a system of topologies over $K$, $V$ is a $K$-variety, and $S \subseteq V(K)$ is a Zariski open subset.  Then $S = U(K)$ for at least one open subvariety $U \subseteq V$.
The inclusion $U \hookrightarrow V$ is an open immersion, so $U(K)$ is a $\cT$-open subset of $V(K)$.
\end{proof}

There is also a finest system of topologies over $K$, the {\bf discrete system} of topologies over $K$ assigning the discrete topology on $V(K)$ to any $K$-variety $V$.

\begin{proposition}
\label{prop:discrete}
Suppose $\cT$ is a system of topologies over $K$ and $\Aa^1(K)$ is $\cT$-discrete.
Then $\cT$ is the discrete system of topologies.
\end{proposition}

We first prove a lemma.

\begin{lemma}
\label{lem:refine-product}
Let $\cT$  be a  system of topologies over $K$, and $V_1, \ldots, V_n$  be $K$-varieties.
Then the $\cT$-topology on  $(V_1 \times \cdots \times V_n)(K)$ refines the product of the $\cT$-topologies on the $V_i(K)$.
\end{lemma}

\begin{proof}
Apply the fact that each projection $(V_1 \times \cdots \times V_n)(K) \to V_i(K)$ is $\cT$-continuous.
\end{proof}

We now prove Proposition~\ref{prop:discrete}.

\begin{proof}
By Lemma~\ref{lem:refine-product} the $\cT$-topology on $\Aa^n(K)$ is discrete for all $n$.
Apply Lemma~\ref{lem:determine}.
\end{proof}

In Lemma~\ref{lem:refine-product}, the topology on $(V_1 \times \cdots \times V_n)(K)$ need not agree with the product topology, even when $\cT$ is the \'etale-open system.  See Section~\ref{section:pseudofinite}.

\subsection{Field topologies}  \label{section:field-topologies}

Aside from the Zariski system, every previously studied system of topologies of which we are aware is induced by a field topology on $K$. 
A \textbf{field topology} on $K$ is a topology $\uptau$ such that inversion $K^\times \to K^\times$ is $\uptau$-continuous and addition and multiplication $K^2 \to K$ are $\uptau$-continuous when $K^2$ is equipped with the product topology.
It is well known and easy to see that a field topology on $K$ is Hausdorff if and only if it is $T_0$.
\emph{We henceforth assume all field topologies are Hausdorff (or equivalently, $T_0$)}.

\meno
Fact~\ref{fact:field-top} below allows us to construct a system of topologies from a field topology; it is essentially proven in \cite[Chapter I.10]{red-book}.

\begin{fact}
\label{fact:field-top}
Suppose that $\uptau$ is a field topology on $K$.
There is a unique system of topologies $\cT_\uptau$ over $K$ such that the $\cT_\uptau$-topology on $\Aa^1(K)$ is $\uptau$ and the $\Sa T_\uptau$-topology on $\Aa^n(K)$ is the product of the $n$ copies of $\uptau$.
\end{fact}

We refer to this as the system of topologies over $K$ \textbf{induced} by $\uptau$ and denote it by $\cT_\uptau$.
Note that if $\cT$ is a system of topologies over $K$ then the $\cT$-topology on $\Aa^1(K)$ is $T_1$, so Fact~\ref{fact:field-top} fails if one allows field topologies to be non-Hausdorff.

\begin{lemma}
\label{lem:field-top-reduce}
Suppose that $\cS$ is a system of topologies over $K$ and $\uptau$ is a field topology on $K$.
Suppose that the $\cS$-topology on $\Aa^1(K)$ refines $\uptau$.
Then $\cS$ refines $\cT_\uptau$.
\end{lemma}

\begin{proof}
By Lemma~\ref{lem:refine-product} the $\cS$-topology on each $\Aa^n(K)$ refines the $\cT_\uptau$-topology.
Apply Lemma~\ref{lem:determine}.
\end{proof}

The following proposition characterizes systems arising from field topologies:

\begin{proposition}
\label{prop:top-product}
Suppose $\cT$ is a system of topologies over $K$.
The following are equivalent:
\begin{enumerate}[leftmargin=*]
\item $\cT$ is induced by a field topology $\uptau$ on $K$.
\item For any $K$-varieties $V,W$ the $\cT$-topology on $(V \times W)(K)$ is the product of the $\cT$-topologies on $V(K)$ and $W(K)$.
\item For any $n$ the $\cT$-topology on $\Aa^n(K)$ is the product of $n$ copies of the $\cT$-topology on $\Aa^1(K)$.
\end{enumerate}
\end{proposition}

\begin{proof}
Fact~\ref{fact:field-top} and the definition of the induced topology show that $(1)$ implies $(2)$.
It is immediate that $(2)$ implies $(3)$.
We show that $(3)$ implies $(1)$.
Suppose $(3)$.
Let $\uptau$ be the $\cT$-topology on $\Aa^1(K)$.
It is easy to see that $\uptau$ is a field topology. 
By $(3)$, the $\cT_\uptau$-topology agrees with the $\cT$-topology on any $\Aa^n(K)$.  By Lemma~\ref{lem:determine}, $\cT_\uptau = \cT$.
\end{proof}

\subsection{Hausdorffness and disconnectedness}
We will need an elementary topological fact whose verification we leave to the reader.

\begin{fact}
\label{fact:inject-disconnect}
Suppose that $S \to T$ is a continuous injection of topological spaces.
If $T$ is Haudorff then $S$ is Hausdorff.
If $T$ is totally separated then $S$ is totally separated.
\end{fact}

A topology on $K$ is \textbf{affine invariant} if it is invariant under affine transformations.  If $\cT$ is a system of topologies, then the $\cT$-topology on $\Aa^1(K) = K$ is affine invariant, by (\ref{st-1}) of Definition~\ref{sys-top:def} in the case where $f$ is an affine transformation.
\begin{lemma}
\label{lem:hausdorff-0}
Suppose that $\uptau$ is an affine invariant topology on $K$.
Then $\uptau$ is Hausdorff if and only if there are two nonempty $\uptau$-open sets with empty intersection.
\end{lemma}
This holds because the action of the affine group on $K$ is two-transitive; we omit the details.

\begin{proposition}
\label{prop:hd}
The following are equivalent for any system $\cT$ of topologies over $K$:
\begin{enumerate}
\item there are disjoint nonempty $\cT$-open subsets $U,V$ of $\Aa^1(K)$,
\item the $\cT$-topology on $\Aa^1(K)$ is Haudorff,
\item the $\cT$-topology on $V(K)$ is Hausdorff for any quasi-projective $K$-variety $V$.
\end{enumerate}
\end{proposition}

\begin{proof}
Lemma~\ref{lem:hausdorff-0} shows that $(1)$ implies $(2)$.
It is clear that $(3)$ implies $(1)$.
We show that $(2)$ implies $(3)$.
Let $V$ be a quasi-projective $K$-variety.
Then there is a morphism $V \to \Pp^n$ such that $V(K) \to \Pp^n(K)$ is injective.
So by Fact~\ref{fact:inject-disconnect} we may suppose that $V = \Pp^n$.
Let $a,b$ be distinct elements of $\Pp^n(K)$.
There is an open immersion $\upiota : \Aa^n \hookrightarrow \Pp^n$ such that $a,b \in \upiota(\Aa^n(K))$.
Hence we may suppose that $V = \Aa^n$.
This case follows by an application of Lemma~\ref{lem:refine-product} and the fact that a product of Hausdorff spaces is Hausdorff.
\end{proof}

We do not know if Proposition~\ref{prop:hd} generalizes to arbitrary varieties.

\medskip
We next discuss total separatedness.
A clopen subset of $S$ is \emph{non-trivial} if it is not $\emptyset$ or $S$.

\begin{proposition}
\label{prop:tot-dis}
The following are equivalent for any system $\cT$ of topologies over $K$:
\begin{enumerate}
\item there is a non-trivial $\cT$-clopen subset of $\Aa^1(K)$,
\item the $\cT$-topology on $\Aa^1(K)$ is totally separated,
\item the $\cT$-topology on $V(K)$ is totally separated for any quasi-projective $K$-variety $V$.
\end{enumerate}
\end{proposition}

Again, we do not know if Proposition~\ref{prop:tot-dis} extends to arbitrary varieties.

\begin{proof}
We work in the $\cT$-topology.
It is clear that $(3)$ implies $(1)$, and $(1)$ implies $(2)$ as the action of the affine group on $\Aa^1(K)$ is $2$-transitive.
We show that $(2)$ implies $(3)$.
We first show that $\Pp^1(K)$ is totally separated.
Fix $\alpha,\alpha^* \in \Pp^1(K)$ with $\alpha \ne \alpha^*$.
Every linear fractional transformation gives a homeomorphism $\Pp^1(K) \to \Pp^1(K)$.
As the action of the group of linear fractional transformations on $\Pp^1(K)$ is $3$-transitive we may suppose that $\infty \notin \{\alpha,\alpha^*\}$.
Let $U = \Pp^1(K) \setminus \{\alpha^*\}$ and $V = \Pp^1(K) \setminus \{\infty\}$.  Then $U$ and $V$ are both homeomorphic to $\Aa^1(K)$.
Choose clopen subsets $O \subseteq U$ and $P \subseteq V$ such that $\alpha \in O$, $\infty \notin O$, $\alpha \in V$, $\alpha^* \notin V$.
It is easy to see that $O \cap P$ is clopen in $\Pp^1(K)$ and $\alpha \in O \cap P$, $\alpha^* \notin O \cap P$.

\meno
We now apply induction to show that $\Pp^n(K)$ is totally separated for every $n \ge 2$.
Fix distinct $\alpha,\alpha^* \in \Pp^n(K)$.
Take $\beta \in \Pp^n(K)$ not lying on the line through $\alpha,\alpha^*$.
We identify the set of lines in $\Pp^n(K)$ through $\beta$ with $\Pp^{n - 1}(K)$.
Let $U = \Pp^n(K) \setminus \{\beta\}$ and define $\pi : U \to \Pp^{n - 1}(K)$ by declaring $\pi(b)$ to be the line through $\beta,b$.
Then $\pi$ is continuous and $\pi(\alpha) \ne \pi(\alpha^*)$.
By induction there is a clopen $O^* \subseteq \Pp^{n - 1}(K)$ such that $\pi(\alpha) \in O^*$, $\pi(\alpha^*) \notin O^*$.
Let $O = \pi^{-1}(O^*)$.
Then $O$ is a clopen subset of $U$ and $\alpha \in O$, $\alpha^* \notin O$.
By the same reasoning there is a clopen $P \subseteq \Pp^n(K) \setminus \{\alpha^*\}$ such that $\alpha \in P$, $\beta \notin P$.
It is easy to see that $O \cap P$ is clopen in $\Pp^n(K)$ and $\alpha \in O \cap P$, $\alpha^* \notin O \cap P$.

\meno
Now suppose that $V$ is a quasi-projective $K$-variety.
Then there is a continuous morphism $V \to \Pp^n$ such that the induced map $V(K) \to \Pp^n(K)$ is injective.
Apply Fact~\ref{fact:inject-disconnect}.
\end{proof}

\subsection{Restriction and extension of systems}
\textbf{Throughout this section $L$ is an extension of $K$, $\cT_K$ is a system of topologies over $K$, and $\cT_L$ is a system of topologies over $L$.} 
We will show that $\cT_L$ ``restricts'' to a system of topologies over $K$ and if $L/K$ is finite then $\cT_K$ ``extends'' to a system of topologies over $L$.
For example the extension of the order topology over $\Rr$ to $\Cc$ will be the complex analytic topology and the restriction of the order topology over $\Rr$ to a subfield $K$ of $\Rr$ will be the order topology over $K$.
 
\subsubsection{Restriction}
\label{section:restriction}
Suppose $V$ is a $K$-variety; recall that $V_L$ is the base change of $V$.
Then $V(K)$ is a subset of $V(L)$ and there is a canonical bijection $V(L) \to V_L(L)$.
So we will consider $V(K)$ to be a subset of $V_L(L)$.
Define the $\rtp(\cT_L)$-topology on $V(K)$ to be the subspace topology induced by the $\cT_L$-topology on $V_L(L)$. 


\begin{proposition}
\label{prop:extension}
$\rtp(\cT_L)$ is a system of topologies over $K$.
\end{proposition}

We call $\rtp(\cT_L)$ the \textbf{restriction} of $\cT_L$ to $K$.

\begin{proof}


Suppose that $f : V \to W$ is a morphism between $K$-varieties $V,W$.
Then the map $f_L: V_L(L) \to W_L(L)$ is $\cT_L$-continuous.
Note that $f$ is the restriction of $f_L$ to $V(K)$.
Hence, $f$ is $ \rtp(\cT_L)$-continuous.
It follows that if $f : V \to W$ is an isomorphism then $f : V(K) \to W(K)$ is a $\rtp(\cT_L)$-homeomorphism.
 
\medskip
Suppose that $f$ is an open immersion. 
We need to show that $V(K) \to W(K)$ is a $\rtp(\cT_K)$-open embedding.
By the remark at the end of the previous paragraph, we can assume $V$ is an open subvariety of $W$ and $f$ is the inclusion.
By Fact~\ref{fact:etale} $f_L : V_L \to W_L$ is an open immersion, so we  identify $V_L$ with an open subvariety of $W_L$.
We identify $W(L)$ with $W_L(L)$ and consider $V_L(L), W(K)$, and $V(K)$ as subsets of $W_L(L)$.
Suppose that $U$ is an $\rtp(\cT_L)$-open subset of $V(K)$.
There is a $\cT_L$-open $U^* \subseteq V_L(L)$ with $U = U^* \cap V(K)$, hence $U^*$ is a $\cT_L$-open subset of $W_L(L)$.
As $U = U^* \cap W(K)$, $U^*$ is a $\rtp(\cT_L)$-open subset of $W(K)$.

\meno
Closed immersions are handled via an identical argument, replacing ``open'' with ``closed.''
\end{proof}

\subsubsection{Restriction of Zariski and field topologies}
\label{section:restrict-zar}

\begin{lemma}
\label{lem:zar-restrict}
Suppose $\cT_K, \cT_L$ are the Zariski systems over $K,L$, respectively.
Then $\rtp(\cT_L)$ agrees with $\cT_K$.
\end{lemma}


\begin{proof}
By Lemmas~\ref{lem:determine} and \ref{lem:refine-Zariski}, it suffices to show that the $\rtp(\cT_L)$-topology on $\Aa^n(K)$ is coarser than the $\Sa T_K$-topology.
Let $Z$ be a Zariski closed set in $L^n$; we must show that $Z \cap K^n$ is Zariski closed.
Fix distinct $f_1,\ldots,f_m \in L[x_1,\ldots,x_n]$ such that $Z$ agrees with $\{ a \in L^n : f_1(a) = \cdots = f_m(a) = 0 \}$.
We easily reduce to the case $m = 1$.  Let $S \subseteq L$ be the $K$-linear span of the coefficients of $f = f_1$.  Let $\beta_1, \ldots, \beta_d$ be a $K$-linear basis of $S$.
Then we can write $f(x_1,\ldots,x_n) = \sum_{i = 1}^d \beta_i g_i(x_1,\ldots,x_n)$ for some $g_1,\ldots,g_d \in K[x_1,\ldots,x_n]$. 
The set $Z \cap K^n$ is cut out by the equations $g_1 = g_2 = \cdots = g_d = 0$, so it is Zariski closed.
\end{proof}

The following lemma shows that our notion of restriction agrees with the usual notion over field topologies.

\begin{lemma}
\label{lem:field-top-restrict}
Suppose that $\uptau$ is a field topology on $L$ and $\sigma$ is the induced subspace topology on $K$.
Then $\sigma$ is a field topology and $\rtp(\cT_\uptau)$ agrees with $\cT_\sigma$.
\end{lemma}

The proof is easy and left to the reader.

\subsubsection{Extension}
\label{section:extension}
Suppose $L/K$ is finite.
Let $[L:K] = m$ and $e_1,\ldots,e_m$ be a $K$-basis for $L$.
Our choice of a basis allows us to identify each $\Aa^n(L)$ with $\Aa^{mn}(K)$.
We show below that there is a system of topologies $\etp(\cT_K)$ on $L$ such that the $\etp(\cT_K)$-topology on every $\Aa^n(L)$ agrees with the $\cT_K$-topology on $\Aa^{mn}(K)$.

\begin{prop-def}
\label{prop:restrict-0}
There is a unique system of topologies $\etp(\cT_K)$ over $L$ such that if $V$ is an affine $L$-variety then the $\etp(\cT_K)$-topology on $V(L)$ agrees with the $\cT_K$-topology on $(\reslk (V) )(K)$ via the natural bijection $V(L) \to (\reslk (V) )(K)$.  
We call $\etp(\cT_K)$ the \textbf{extension} of $\cT_K$ to $L$.
\end{prop-def}

Proposition~\ref{prop:restrict-0} follows from the first two claims of Fact~\ref{fact:restrict-0} and Lemma~\ref{lem:affine}.

%

\begin{remark}
\label{remark:weil-restrict}
The Weil restriction of an arbitrary $L$-variety need not exist, so Proposition~\ref{prop:restrict-0} involves affine varieties instead of arbitrary varieties.
However, it can be easily checked that if the Weil restriction of an $L$-variety $V$ exists, then the $\etp(\cT_K)$-topology on $V(L)$ agrees with the $\cT_K$-topology on $(\reslk (V))(K)$.
\end{remark}

\section{The \'etale-open topology: proof of Theorem A}
\label{section:thmA}

In this section we prove Theorem A; see Lemmas~\ref{lem:intersection}, \ref{lem:cts}, \ref{lem:open-map}, and Proposition~\ref{prop:closed-embedd} below.  We will use the basic properties of \'etale morphisms listed in Fact~\ref{fact:etale}.




\meno
Recall that an \textbf{\'etale image in} $V(K)$ is a set of the form $f(X(K))$ for some \'etale morphism $f : X \to V$ of $K$-varieties.
When we wish to keep track of the underlying field we will write ``$K$-\'etale image in $V(K)$".

\begin{lemma}
\label{lem:intersection}
Let $V$ be a $K$-variety.
The collection of \'etale images in $V(K)$ contains every Zariski open subset of $V(K)$ and is closed under finite unions and finite intersections.
So the collection of \'etale images in $V(K)$ is a basis for a topology refining the Zariski topology.
\end{lemma}

\begin{proof}
Suppose $U \subseteq V(K)$ is Zariski open.
Then there is an open subvariety $O$ of $V$ such that $U =  O(K)$.
The inclusion $O \hookrightarrow V$ is an open immersion, hence \'etale.  Thus $U$ is an \'etale image.  Finite unions and intersections were handled in Lemma~\ref{intersection:lem2a}.
\end{proof}

For each $K$-variety $V$, let $\cE_V$ be the topology on $V(K)$ with basis the \'etale images in $V(K)$. We call $\cE_V$ the \textbf{\'etale-open topology on $V$}.
We briefly give a more elementary definition of the \'etale open topology on $K^n$.
Let $\mathcal{P}$ be the set of $(f,g)$ such that $f,g \in K[x_1,\ldots,x_n,y]$, $f$ is monic in $y$, and the following equivalent conditions hold (here $\kalg$ is an algebraic closure of $K$ and $L$ is a field):
\begin{enumerate}[leftmargin=*]
\item $\der f/\der y$ does not vanish on the subvariety of $\Aa^n \times \Aa^1$ given by $f = 0 \ne g$.
\item $\der f / \der y$ is invertible in $\left(K[x_1,\ldots,x_n,y]/(f)\right)[1/g]$.
\item If $L$ extends $K$ and $(\alpha,\beta) \in L^n\times L$ satisfies $f(\alpha,\beta) = 0 \ne g(\alpha,\beta)$ then $(\der f / \der y)(\alpha,\beta) \ne 0$.
\item If $(\alpha,\beta) \in (\kalg)^n\times\kalg$ satisfies $f(\alpha,\beta) = 0 \ne g(\alpha,\beta)$ then $(\der f / \der y)(\alpha,\beta) \ne 0$.
\end{enumerate}
Then the collection of sets of the form $\{ \alpha \in K^n : \exists \beta \in K \hspace{4pt} [f(\alpha,\beta) = 0 \ne g(\alpha,\beta)] \}$ for $(f,g) \in \mathcal{P}$ forms a basis for the \'etale open topology on $K^n$.
This is immediate from the definitions and the fact that every \'etale morphism is locally standard \'etale up to isomorphism.
We omit the details as we will not require this at present.

\medskip
We will show that the collection $(\cE_V)_{V\in \Var_K}$ is a system of topologies such that $(V(K),\cE_V) \to (W(K),\cE_W)$ is an open map for any \'etale morphism $V \to W$.  Specifically,
for any $K$-variety morphism $f : V \to W$, we will show the following:
\begin{enumerate}
\item $f: (V(K), \cE_V) \to (W(K), \cE_W)$ is continuous.
\item If $f$ is \'etale then $f: (V(K),\cE_V) \to (W(K), \cE_W)$ is an open map.
\item If $f$ is a closed immersion then $f: (V(K), \cE_V) \to ( W(K), \cE_W )$ is a closed embedding.
\end{enumerate}
Note that open immersions are \'etale (Fact~\ref{fact:etale}.1), so (2) generalizes condition (\ref{st-2}) of Definition~\ref{sys-top:def}.
We first establish (1) above.

\begin{lemma}
\label{lem:cts}
Suppose $V \to W$ is a morphism of $K$-varieties and let $f : V(K) \to W(K)$ be the induced map.
If $U$ is an \'etale image in $W(K)$ then $f^{-1}(U)$ is an \'etale image in $V(K)$.
Hence, $f : ( V(K), \cE_V ) \to ( W(K), \cE_W )$ is continuous.
\end{lemma}

\begin{proof}
The second statement is immediate from the first so we only treat the first.
Suppose $h : X \to W$ is an \'etale morphism of $K$-varieties and $U = h(X(K))$.
By Fact \ref{fact:etale}.3, the map $g: X \times_{W} V \to V$ given by the pullback square is \'etale.
As the $K$-points functor preserves pullback squares the diagram
\[ \xymatrix{
(V \times_{W} X)(K) \ar[r] \ar[d]^{g} & X(K) \ar[d]^{h} \\ V(K)
\ar[r]^{f} & W(K)}\]
is also pullback square.
Therefore, $(f)^{-1}(U) = g((V \times_{W} X)(K))$ is an \'etale image.
\end{proof}

Next, we verify (2).

\begin{lemma}
\label{lem:open-map}
Suppose  $f : V \to W$ is an \'etale morphism of $K$-varieties.
If $U$ is an \'etale image in $V(K)$ then $f(U)$ is an \'etale image in $W(K)$.
Hence, $f : ( V(K), \cE_V ) \to ( W(K), \cE_W )$ is an open map.
\end{lemma}

\begin{proof}
Let $X$ be a $K$-variety, $h : X \to V$ be an \'etale morphism, and $U = h(X(K))$.
By Fact~\ref{fact:etale}, $f \circ h : X \to W$ is \'etale, so $f(U) = (f \circ h)(X(K))$ is an \'etale image.
\end{proof}




Now we check (3).
\begin{proposition}
\label{prop:closed-embedd}
Let $V,W$ be $K$-varieties and $f : V \to W$ be a closed immersion.
Then every  \'etale image in $V(K)$ is the inverse image under $f$ of an \'etale image in $W(K)$.
Therefore $f : ( V(K), \cE_V) \to ( W(K), \cE_W )$ is a closed embedding.
\end{proposition}

\begin{proof}
Without loss of generality, we may assume that $V$ is a closed subvariety of $W$ and $f$ is the inclusion map.
Then $V(K)$ is a closed subset of $W(K)$ in the \'etale-open topology by Lemma~\ref{lem:intersection}. As $f$ is continuous with respect to the \'etale-open topology,  the \'etale-open topology on $V(K)$ refines the induced subspace topology.
It remains to show that the subspace topology on $V(K)\subseteq (W(K),\cE_W)$ refines the \'etale-open topology on $V(K)$.
Suppose that $U$ is an \'etale image in $V$ which is given by an \'etale map $g: X \to V$.
Let $p$ be a scheme-theoretic point of $X$.
Then by Fact~\ref{fact:ega}, we get a Zariski open neighborhood $O_p$ of $p$ in $X$ and an \'etale morphism $h_p: Y_p \to W$, such that $O_p$ is isomorphic to $Y_p \times_W V$ as $V$-schemes.
For each $p$, let $U'_p \subseteq W(K)$ be the image of $Y_p(K)$ under $h_p$.
By compactness of the Zariski topology on $X$, there is a finite set $\Delta \subseteq X$ such that $(O_p)_{p \in \Delta}$ forms a finite cover over $X$.
Then $(U'_p \cap V(K))_{p \in \Delta}$ is a finite cover of $U$. By Lemma~\ref{lem:intersection}, $\bigcup_{p \in \Delta} U'_p$ is an \'etale image in $W$, so we get the first statement.
The second statement is immediate.
\end{proof}

This completes the proof of Theorem A---the collection $(\cE_V)_{V\in \Var_K}$ is a system of topologies.  We call this \textbf{the \'etale-open system of topologies} on $K$.

\medskip
If $\cT$ is a system of topologies over $K$, say that $\cT$ \textbf{turns \'etale morphisms into open maps} if, for any \'etale-morphism $V \to W$, the map $V(K) \to W(K)$ is $\cT$-open.  By Lemma~\ref{lem:open-map}, the \'etale-open system of topologies turns \'etale morphisms into open maps.

\begin{proposition} \label{characterization}
  The \'etale-open system of topologies is the coarsest system of topologies which turns \'etale morphisms into open maps.
\end{proposition}
\begin{proof}
   Suppose $\cT$ turns \'etale morphisms into open maps.  We claim that $\cT_V$ refines $\cE_V$ for each $K$-variety $V$.  It suffices to show that every \'etale-image in $V(K)$ is $\cT$-open.  Let $f : W \to V$ be \'etale.  Then $W(K) \to V(K)$ is a $\cT$-open map, so its image is $\cT$-open.
\end{proof}
Proposition~\ref{characterization} is the original reason for the name ``\'etale-open topology.''

\begin{lemma}
\label{lem:affine2}
Suppose that $\cT$ is a system of topologies over $K$.
If $V(K) \to W(K)$ is $\cT$-open for any \'etale morphism $V \to W$ of affine $K$-varieties then $V(K) \to W(K)$ is $\cT$-open for any \'etale morphism $V \to W$ of $K$-varieties.
\end{lemma}
The proof is similar to that of Lemma~\ref{lem:affine}.

\subsection{Extension and restriction of the \'etale-open system of topologies}
We prove two useful results about extension and restriction of the \'etale-open system of topologies.

\begin{proposition}
\label{prop:restrict-1}
Suppose that $L/K$ is finite.
Then $\etp(\cE_K)$ refines $\cE_L$.
\end{proposition}

\begin{proof}
By Fact~\ref{fact:restriction-etale} and Lemma~\ref{lem:affine2} $V(L) \to W(L)$ is an $\etp(\cE_K)$-open map for any \'etale morphism $V \to W$ of $L$-varieties.
Therefore any $L$-\'etale image is $\etp(\cE_K)$-open.
\end{proof}

Suppose $L$ is an algebraic extension of $K$.
We show below that $\cE_K$ is the discrete system of topologies if and only if $K$ is not large.
So by Fact~\ref{fact:large-fin} we see that if $\cE_L$ is discrete then $\cE_K$ is discrete.
The following theorem is therefore is a topological refinement of Fact~\ref{fact:large-fin}.

\begin{theorem}
\label{thm:general}
Suppose that $L/K$ is algebraic.
Then $\cE_K$ refines $\rtp(\cE_L)$.
\end{theorem}

We first prove a lemma.

\begin{lemma}
\label{lem:basic-down}
Suppose that $L$ is a finite extension of $K$, $V$ is a $K$-variety such that $\reslk(V_L)$ exists and $U \subseteq V_L(L)$ is an $L$-\'etale image.
Then $U \cap V(K)$ is a $K$-\'etale image.
In particular if $U \subseteq \Aa^n_L(L)$ is an $L$-\'etale image then $U \cap \Aa^n(K)$ is a $K$-\'etale image.
\end{lemma}

\begin{proof}
Let $X$ be an $L$-variety and $h : X \to V_L$ be an \'etale morphism such that $U = h(X(L))$.
Let $X_1,\ldots,X_n$ be affine open subvarieties of $X$ covering $X$.
Let $X^*$ be the disjoint union of $X_1,\ldots,X_n$ and $h^*$ be the natural morphism $X^* \to V_L$.
Then $X^*$ is affine, $h$ is \'etale, and $U = h^*(X^*(L)))$.  Replacing $X$ with $X^*$, we may assume that $X$ is affine.
This ensures that $\reslk(X)$ exists.
Let $g = \reslk(h)$; this is an \'etale morphism $\reslk(X) \to \reslk(V_L)$ by Fact~\ref{fact:restrict-0}.
As before we identify $(\reslk(V_L))(K)$ and $V_L(L)$.
Then $U$ is a $K$-\'etale image in $(\reslk(V_L))(K)$.
By Fact~\ref{fact:can-is-closed} the natural morphism $V \to \reslk(V_L)$ is a closed immersion.
By Proposition~\ref{prop:closed-embedd}, $U \cap V(K)$ is a $K$-\'etale image in $V(K)$.
\end{proof}

We now prove Theorem~\ref{thm:general}.

\begin{proof}
The case when $L/K$ is finite follows by Lemma~\ref{lem:determine} and Lemma~\ref{lem:basic-down}.
Suppose that $L/K$ is infinite.
By Lemma~\ref{lem:determine} it suffices to fix an affine $K$-variety $V$ and show that the $\cE_K$-topology on $V(K)$ refines the $ \rtp(\cE_L)$-topology.
Let $X$ be an $L$-variety and $g : X \to V_L$ be an \'etale morphism.
We show that $U = g(X(L)) \cap V(K)$ is $\cE_K$-open.
Let $K \subseteq J \subseteq L$ be a finite extension such that $X$ and $g$ are defined over $J$.
So there is a $J$-variety $Y$ and a morphism $f : Y \to V_J$ such that $X = Y_L$ and $g = f_L$.
Then $U$ is the union of the $f_M(Y_M(M)) \cap V(K)$ where $M$ ranges over finite extensions $J \subseteq M \subseteq L$.
By the finite case each $f_M(Y_M(M)) \cap V(K)$ is $\cE_K$-open.
\end{proof}

We will use the following proposition to show that the \'etale-open topology over a separably closed field agrees with the Zariski topology.

\begin{proposition}
\label{prop:inseparable}
Suppose that $L/K$ is purely inseparable.
Then $\cE_K$ agrees with $\rtp(\cE_L)$.
\end{proposition}

We let $\mathbf{k}(a)$ be the residue field of a scheme-theoretic point $a$ on a $K$-variety $V$.

\begin{proof}
A purely inseparable extension is algebraic so by Theorem~\ref{thm:general} it suffices to show that $\rtp(\cE_L)$ refines $\cE_K$.
Let $f : X \to V$ be an \'etale morphism of $K$-varieties and $U = f(X(K))$.
We show that $U$ is $\rtp(\cE_L)$-open.
The base change $f_L : X_L \to V_L$ is \'etale by Fact~\ref{fact:etale}.
Let $U' = f_L(X_L(L))$.
It suffices to show that $U = U' \cap V(K)$.
We have $U \subseteq U' \cap V(K)$ so it is enough to fix $a \in U' \cap V(K)$ and show that $a \in U$.
Then $a$ is a point in $V$ with $\mathbf{k}(a)= K$. 
Let $f_{a} : X_{a} \to a$ be the scheme-theoretic fiber of $X$ over $a$. 
Since $a$ is in $U'$ and we identify $X_L(L)$ and $X(L)$, there is a (scheme-theoretic) point $b\in X_a$ with $\mathbf{k}(b)$ embeddable into $L$. Furthermore $f_a$ is \'etale as it is the base change of an \'etale morphism.
Hence, $\mathbf{k}(b)$ is a separable extension of $K$. The only separable extension of $K$ in $L$ is $K$, so $\mathbf{k}(b) =K$. Thus $b$ is a $K$-point of $X$ which implies that $a$ is in $U$. 
\end{proof}

\section{Classical examples and generalizations}
\label{sec:example}
 
In this section we show that the \'etale-open topology agrees with known topologies on separably closed fields and t-Henselian fields.
This covers many natural examples.

\subsection{Separably closed fields}
\label{section:separably-closed}

 
By Lemma~\ref{lem:refine-Zariski} the \'etale-open topology always refines the Zariski topology.
The converse holds when $K$ is separably closed.

\begin{proposition}
\label{prop:zar}
Suppose $K$ is separably closed.
Then the \'etale-open topology on $V(K)$ agrees with the Zariski topology for any $V$.
\end{proposition}

\begin{proof}
First suppose that $K$ is algebraically closed.
By Lemma~\ref{lem:refine-Zariski} it suffices to show that the Zariski topology refines $\cE_V$.
Let $f : X \to V$ be an \'etale morphism of $K$-varieties and $U = f(X(K))$.
We show that $U$ is Zariski open in $V(K)$.
The
image of $f : X \to V$ (the map on scheme-theoretic points) is an open subset $O$ of $V$ by Fact~\ref{fact:etale}.4, and $O$ naturally carries the structure of an open subvariety of $V$.
As $K$ is algebraically closed $U = O(K)$.

\meno
Now suppose that $K$ is separably closed.
Fix an algebraic closure $L$ of $K$.
Then $\cE_L$ is the Zariski topology over $L$.
Furthermore, $L$ is a purely inseparable extension of $K$, so $\cE_K$ agrees with $\rtp( \cE_L)$ by Proposition~\ref{prop:inseparable}.
Apply Lemma~\ref{lem:zar-restrict}.
\end{proof}

\subsection{t-Henselian fields}
\label{section:t-Henselian}
We first recall some definitions.

\begin{definition}
\label{defn:V}
Let $\uptau$ be a non-discrete field topology on $K$.
A subset $B$ of $K$ is said to be \textbf{$\uptau$-bounded} if for every open neighborhood $U$ of zero there is non-zero $a \in K$ such that $a \cdot B \subseteq U$.
It is easy to see that any finite subset of $K$ is bounded and the collection of bounded sets is closed under finite unions, subsets, and additive and multiplicative translates. 
Then $\uptau$ is  a \textbf{V-topology} if $(K \setminus U)^{-1}$ is bounded for any neighborhood $U$ of zero.
\end{definition}

It is easy to see that the topology on $K$ associated to any field order, non-trivial valuation, or non-trivial absolute value is a V-topology.  The following converse is due to Kowalsky and D\"urbaum, and independently, Fleischer.
\begin{fact}[{\cite[Theorem B.1]{EP-value}}]
\label{fact:V}
Suppose that $\uptau$ is a field topology on $K$.
If $\uptau$ is a V-topology then $\uptau$ is induced by some absolute value or valuation on $K$.
\end{fact}

We now recall the topological analogue of Henselianity.

\begin{definition}
\label{defn:t-hensel}
Let $\uptau$ be a field topology on $K$.
Then $\uptau$ is \textbf{t-Henselian} if
\begin{enumerate}
\item $\uptau$ is a V-topology, and
\item for any $n$ there is a $\uptau$-open neighborhood $U$ of $0$ such that if $\alpha_0,\ldots,\alpha_{n} \in U$ then $x^{n + 2} + x^{n + 1} + \alpha_{n} x^{n} + \cdots + \alpha_1 x + \alpha_0$ has a root in $K$.
\end{enumerate}
If $K$ admits a t-Henselian topology then we say that $K$ is t-Henselian.
\end{definition}

The concept of t-henselianity is due to Prestel and Ziegler~\cite{Prestel1978}.
A valuation topology induced by a non-trivial Henselian valuation is t-Henselian, and so is the order topology on a real closed field.
Prestel and Ziegler show that if $K$ is elementarily equivalent to a t-Henselian field and is not separably closed then $K$ is t-Henselian and if $K$ is $\aleph_1$-saturated and t-Henselian then $K$ admits a non-trivial Henselian valuation.
So a non-separably closed field is t-Henselian if and only if it is elementarily equivalent to a Henselian field.
They also show that a non-separably closed field admits at most one t-Henselian topology.
So if $K$ is not separably closed then we refer to the unique t-Henselian topology on $K$ as \textbf{the t-Henselian topology}.

\begin{theorem}
\label{thm:hensel-main}
Suppose that $K$ is t-Henselian and not separably closed.
Then the \'etale-open system of topologies over $K$ is induced by the t-Henselian topology.
\end{theorem}

We will need several lemmas.
We let $\Sa T_\uptau$ be the system of topologies induced by a field topology $\uptau$ on $K$.

\begin{fact}
\label{fact:open-map}
Suppose that $\uptau$ is a t-Henselian topology on $K$, $V$ and $W$ are $K$-varieties, and $f : V \to W$ is an \'etale morphism.
Then $f : V(K) \to W(K)$ is $\cT_\uptau$-open.
\end{fact}
Fact~\ref{fact:open-map} is essentially \cite[Proposition~2.8]{HHJ-henselian}.
They only state the result for perfect fields, but their proof goes through exactly as written without the assumption of perfection.

\medskip
Fact~\ref{fact:open-map} and Proposition~\ref{characterization} show that $\cT_\uptau$ refines $\cE_K$ when $\uptau$ is t-Henselian. 
We now prove some general lemmas which will be used to show the other direction: $\cE_K$ refines $\cT_\uptau$.

\begin{lemma}
\label{lem:def-loc-bound}
Suppose that $\uptau_0$ is an affine invariant topology on $K$, $\uptau_1$ is a non-discrete field topology on $K$, and some non-empty $X \subseteq K$ is $\uptau_0$-open and $\uptau_1$-bounded.
Then $\uptau_0$ refines $\uptau_1$.
\end{lemma}

\begin{proof}
After translating if necessary we suppose that zero is in $X$.
Let $U$ be a nonempty $\uptau_1$-open subset of $K$.
So for every $\beta \in U$ there is $\alpha_\beta \in K^\times$ such that $\alpha_\beta X \subseteq U - \beta$.
As $\uptau_0$ is affine invariant $\alpha_\beta X + \beta$ is $\uptau_0$-open for all $\beta \in U$.
So $U = \bigcup_{\beta \in U} (\alpha_\beta X + \beta)$ is $\uptau_0$-open.
\end{proof}

\begin{lemma}
\label{lem:V-proj}
Let $\upiota : \Aa^1 \hookrightarrow \Pp^1$ be the open immersion given by $\upiota(x) = (x:1)$.
Suppose that $\uptau$ is a V-topology on $K$.
Then a subset $X$ of $\Aa^1(K)$ is $\uptau$-bounded if and only if $\infty$ is not in the $\cT_\uptau$-closure of $\upiota(X)$.
\end{lemma}

We let $\operatorname{Cl}(X)$ be the $\Sa T_\uptau$-closure of $X \subseteq \Pp^1(K)$.

\begin{proof}
As $\uptau$ is a V-topology, $X$ is $\uptau$-bounded if and only if zero is not in the closure of $(X \setminus \{0\})^{-1}$.
Observe that $0$ is not in $\operatorname{Cl}((X \setminus \{0\})^{-1})$ if and only if $\infty$ is not in $\operatorname{Cl}(\upiota(X))$.
\end{proof}

\begin{lemma} 
\label{dense-refiner:lem}
Suppose that $\Sa S$ is a system of topologies over $K$ and $\uptau$ is a $V$-topology on $K$.
Suppose some non-empty $\Sa S$-open $U \subseteq \Aa^1(K)$ is not $\uptau$-dense in $\Aa^1(K)$.
Then $\Sa S$ refines $\Sa T_\uptau$.
\end{lemma}

\begin{proof}
By Lemmas~\ref{lem:field-top-reduce} and \ref{lem:def-loc-bound} it suffices to  produce a subset of $\Aa^1(K)$ which is $\Sa S$-open and $\uptau$-bounded.
Fix $\alpha \in \Aa^1(K)$ such that $\alpha$ does not lie in the $\uptau$-closure of $U$.
We let $f :\Pp^1(K) \to \Pp^1(K)$ be given by $f(x) = 1/(x - \alpha)$.
Both the $\Sa S$-topology and the $\cT_\uptau$-topology on $\Pp^1(K)$ are invariant under linear fractional transformations, so $f(U)$ is $\Sa S$-open in $\Pp^1(K)$ and $\infty = f(\alpha)$ is not in the $\cT_\uptau$-closure of $V$.
By Lemma~\ref{lem:V-proj} $f(U) \subseteq \Aa^1(K)$ is $\uptau$-bounded.
\end{proof}

\begin{lemma}
\label{lem:key-refine}
Suppose that $\uptau$ is a V-topology on $K$.
Suppose $f \in K[x]$ is such that $f' \not\equiv 0$ and $f(\Aa^1(K))$ is not $\uptau$-dense in $\Aa^1(K)$.
Then $\cE_K$ refines $\cT_\uptau$.
\end{lemma}

\begin{proof}
Let $U$ be the open subvariety of $\Aa^1$ given by $f'(x) \ne 0$.
So $f$ gives an \'etale morphism $U \to \Aa^1$.
Then $f(U(K))$ is an \'etale image in $\Aa^1(K)$ and $f(U(K))$ is not $\uptau$-dense in $\Aa^1(K)$.
An application of Lemma~\ref{dense-refiner:lem} shows that $\cE_K$ refines $\cT_\uptau$.
\end{proof}

Fact~\ref{fact:prestel-def} is \cite[Lemma 7.5]{Prestel1978}.

\begin{fact}
\label{fact:prestel-def}
Suppose that $\uptau$ is a t-Henselian topology on $K$ and $f \in K[x]$ is separable and has no zeros in $K$.
Then zero is not a $\uptau$-limit point of $f(\Aa^1(K))$.
\end{fact}

We now prove Theorem~\ref{thm:hensel-main}.

\begin{proof}
Let $\uptau$ be the t-Henselian topology on $K$.
By Fact~\ref{fact:open-map}, $\cT_\uptau$ refines $\cE_K$.
As $K$ is not separably closed there is a nonconstant irreducible separable $f \in K[x]$.
Then $f' \not\equiv 0$.
By Fact~\ref{fact:prestel-def}, $f(\Aa^1(K))$ is not $\uptau$-dense.
By Lemma~\ref{lem:key-refine}, $\cE_K$ refines $\cT_\uptau$.
\end{proof}

Theorem~\ref{thm:hensel-main} is the first claim of Theorem B.  The second claim will be proven in \S\ref{section:thmb-redux}.

\begin{corollary}
\label{rcf-case} Suppose that $K$ is real closed.  Then the \'etale-open system of topologies over $K$ is induced by the order topology on $K$.
\end{corollary}

\begin{corollary}
\label{hens-case} Suppose that $(K,v)$ is a henselian valued field, and $K$ is not separably closed.  Then the \'etale-open system of topologies over $K$ is induced by the valuation topology.
\end{corollary}

\subsection{Field orders and valuations}
\label{section:order-val}

In this section we show that the \'etale-open topology on a general field refines several other kinds of field topologies.
We first handle order topologies.

\begin{proposition}
\label{prop:order-refine}
Suppose $<$ is a field order on $K$.
Then $\cE_K$ refines the $<$-topologies.
\end{proposition}

Note that the $<$-topology is a V-topology.

\begin{proof}
Let $f \in K[x]$ be $f(x) = x^2$.
As $K$ is ordered $\Chara(K) = 0$ so $f' \not\equiv 0$.
The set of squares is not $<$-dense in $\Aa^1(K)$.
Apply Lemma~\ref{lem:key-refine}.
\end{proof}

Given a valued field $(L,v)$ we let $\Gamma_v$ be the value group, $\bk_v$ be the residue field, and $\cT_v$ be the induced system of topologies over $L$.
We refer to \cite[Chapter 5]{EP-value} for an account of the Henselization of a valued field.

\begin{theorem}
\label{thm:refine-valuation}
Suppose that $v$ is a non-trivial valuation on $K$ and the Henselization of $(K,v)$ is not separably closed.
Then $\cE_K$ refines $\cT_v$.
\end{theorem}

We first prove a lemma.

\begin{lemma}
\label{lem:immediate}
Suppose that $v$ is a valuation on $K$ and $(L,w)$ is an extension of $(K,v)$.
Then $\rtp(\cT_w)$ refines $\cT_v$.
If $\Gamma_w = \Gamma_v$ then $\cT_v$ agrees with $\rtp(\cT_w)$.
\end{lemma}

\begin{proof}
Let $\uptau$ be the subspace topology on $K$ induced by $w$.
Then $\uptau$ is a field topology.
By Lemma~\ref{lem:field-top-restrict}, $\cT_\uptau$ agrees with $\rtp(\cT_w)$.
As $w$ is an extension of $v$, $\uptau$ refines the $v$-topology on $K$.
By Lemma~\ref{lem:field-top-reduce}, $\cT_\uptau$ refines $\cT_v$.
Now suppose that $\Gamma_v = \Gamma_w$.
We show that $\cT_v$ refines $\cT_\uptau$.
By Lemma~\ref{lem:field-top-reduce}, it suffices to show that the $v$-topology on $K$ refines $\uptau$.
Fix $a \in L$, $\gamma \in \Gamma_w$, and let $B$ be the $w$-ball $\{a^* \in L : w(a - a^*) > \gamma\}$.
It suffices to show that $B \cap K$ is $v$-open.
We may suppose that $B \cap K$ is nonempty and fix $b \in B \cap K$.
By the ultrametric triangle inequality $B \cap K$ is the set of $b^* \in K$ such that $v(b - b^*) > \gamma$.
As $\gamma \in \Gamma_v$ we see that $B \cap K$ is a $v$-ball, hence $v$-open.
\end{proof}

We now prove Theorem~\ref{thm:refine-valuation}.

\begin{proof}
Let $(L,w)$ be the Henselization of $(K,w)$.
Then $\cE_L$ agrees with $\cT_w$ as $w$ is a non-trivial Henselian valuation on $L$.
As $(L,w)$ is an immediate extension of $(K,v)$, $\rtp(\cE_L)$ agrees with $\cT_v$ by Lemma~\ref{lem:immediate}.
As $L/K$ is algebraic, $\cE_K$ refines $\rtp(\cE_L)$ by Theorem~\ref{thm:general}.
\end{proof}

\begin{corollary}
\label{cor:refine-valuation}
Suppose $v$ is a non-trivial valuation on $K$ and either
\begin{enumerate}
\item $\Gamma_v$ is not divisible, or
\item $\bk_v$ is not algebraically closed.
\end{enumerate}
Then $\cE_K$ refines $\cT_v$.
\end{corollary}

In particular $\cE_K$ refines $\cT_v$ for any non-trivial discrete valuation $v$ on $K$.

\begin{proof}
Suppose that $\cE_K$ does not refine $\cT_v$ and let $(L,w)$ be the Henselization of $(K,v)$.
Then we have $\Gamma_w = \Gamma_v$ and $\bk_w = \bk_v$.
By Theorem~\ref{thm:refine-valuation}, $L$ is separably closed.
A non-trivial valuation on a separably closed field has divisible value group and algebraically closed residue field~\cite[3.2.11]{EP-value}.
\end{proof}

In characteristic zero Corollary~\ref{cor:refine-valuation} is equivalent to Theorem~\ref{thm:refine-valuation} as a characteristic zero Henselian valued field is algebraically closed if and only if it has divisible value group and algebraically closed residue field.
This fails in positive characteristic.
The field $\palg\langle\langle t \rangle\rangle$ of Puiseux series over the algebraic closure $\palg$ of $\Ff_p$ is Henselian, the canonical valuation on $\palg\langle\langle t \rangle\rangle$ has value group $(\Qq,+)$ and residue field $\palg$, and $\palg\langle\langle t \rangle\rangle$ is \emph{not} separably closed.

\meno
We conclude by describing a proof of Proposition~\ref{prop:order-refine} along the lines of Theorem~\ref{thm:refine-valuation}.
Suppose that $<$ is a field order on $K$.
Let $L$ be the real closure of $(K,<)$.
By Corollary~\ref{rcf-case} $\cE_L$ is induced by the order topology over $L$.
As $L/K$ is algebraic $\cE_K$ refines $\rtp(\cE_L)$ and $\rtp
(\cE_L)$ is induced by the order topology on $K$ by Lemma~\ref{lem:field-top-restrict}.

\section{Field-theoretic versus topological properties}
\label{sec:topo-field}
 
We relate topological properties of $\cE_K$ to algebraic properties of $K$.

\subsection{Discreteness}
The \'etale-open topology is only non-trivial over large fields.

\begin{theorem}
\label{thm:3-strong}
The following are equivalent
\begin{enumerate}
\item $K$ is not large.
\item $\cE_K$ is the discrete system of topologies.
\item There is a $K$-variety $V$ such that $V(K)$ is infinite and $(V(K),\cE_K)$ is discrete.
\item There is an irreducible $K$-variety $V$ of dimension $\ge 1$ and a smooth $p \in V(K)$ such that $p$ is $\cE_K$-isolated in $V(K)$.
\end{enumerate}
\end{theorem}

This generalizes Lemma~\ref{large:lem}.  
The proof requires a few lemmas.

\begin{fact}
\label{fact:large}
Suppose that $K$ is large and $V$ is a smooth irreducible $K$-variety of dimension at least one.
If $V(K)$ is nonempty then $V(K)$ infinite.
\end{fact}

See \cite[Proposition~2.6]{Pop-little} for a proof.

\medskip
If $V$ is a $K$-variety and $S \subseteq V(K)$, say that $S$ is \textbf{Zariski dense} in $V$ if every non-empty open subvariety $U \subseteq V$ has $U(K) \cap S \ne \emptyset$.
\begin{remark} \label{z-closure}
Let $V$ be a $K$-variety and $S$ be a subset of $V(K)$.
Let $t : V(K) \to V$ be the map sending a $K$-point to the corresponding scheme-theoretic point.
Let $t(S) \subseteq V$ be the image, and let $W$ be the closure $\overline{t(S)}$ with the reduced induced subvariety structure.
Then $W$ is the smallest closed subvariety of $V$ such that $W(K) \supseteq S$.
We call $W$ the \textbf{Zariski closure} of $S$.
Note that $W$ is reduced and $S$ is Zariski dense in $W(K)$.
\end{remark}

\begin{lemma}
\label{K-point dense}
Let $V$ be a $K$-variety and suppose that $V(K)$ iz Zariski dense in $V$.
Then $U(K)$ is Zariski dense in $U$ for every irreducible component $U$ of $V$.
\end{lemma}

\begin{proof}
Let $U_1,...,U_n$ be the distinct irreducible components of $W$, so in particular, $U_i\cap U_j\subsetneq U_i$ for all $i,j$.
Suppose that $U_1(K)$ is not Zariski dense in $U_1$.
Then there is a proper Zariski-closed subset $U^*_1$ of $U_1$ containing $U_1(K)$.
Then $V(K)$ is a subset of $U^*_1 \cup U_2 \cup \cdots \cup U_n$. As $U^*_1$ and each $U_i$ is Zariski closed set, we have $V = U^*_1 \cup U_2 \cup \cdots \cup U_n$.
Then $U_1$ is a subset of $U^*_1 \cup (U_1 \cap U_2) \cup \cdots \cup (U_1 \cap U_n)$.
This contradicts irreducibility of $U_1$.
\end{proof}

\begin{lemma} \label{smooth-dimension}
Let $V$ be an $n$-dimensional reduced $K$-variety with $V(K)$ Zariski dense in $V$.
Then there is a non-empty open subvariety $U \subseteq V$ such that $U$ is smooth of dimension $n$.
\end{lemma}

Recall that $V_\mathrm{sm}, V_\mathrm{reg}$ is the smooth, regular locus of a $K$-variety $V$, respectively.

\begin{proof}
By Fact~\ref{fact:smooth points} $V_\mathrm{sm}$ and $V_\mathrm{reg}$ are open subvarieties of $V$.
We claim that $\dim V_\sm$ agrees with $\dim V$.
If not, there is a non-empty open subvariety $U \subseteq V$ such that $U \cap V_\sm = \emptyset$.
By Fact~\ref{fact:smooth points} $U_\reg$ is non-empty.
Then $U_\reg(K) \ne \emptyset$ as $V(K)$ is Zariski dense in $V$.
Take $p \in U_\reg(K)$.  Then $p$ is a regular point of $V$ as well.
By Fact~\ref{fact:smooth points}, $V_\reg(K) = V_\sm(K)$.  Therefore, $p \in V_\sm(K)$.  Then $V_\sm \cap U \ne \emptyset$, a contradiction.
This shows that $\dim V_\sm = \dim V = n$.
One of the connected components of $V_\sm$ is smooth of dimension $n$.
\end{proof}

We now prove Theorem~\ref{thm:3-strong}.
\begin{proof}
$(1) \Rightarrow (2)$.
Suppose that $K$ is not large.
By Proposition~\ref{prop:discrete} it is enough to show that the \'etale-open topology on $\Aa^1(K)$ is discrete.
As the action of the affine group on $\Aa^1(K)$ is transitive it suffices to show that some singleton subset of $\Aa^1(K)$ is $\cE_K$-open.
As the $\cE_K$-topology on $\Aa^1(K)$ is $T_1$ it suffices to produce a nonempty finite $\cE_K$-open subset of $\Aa^1(K)$.
Let $X$ be a smooth $K$-variety of dimension one such that $X(K)$ is finite and nonempty.
Fix $p \in X(K)$.
Let $f \in \mathcal{O}_p$ be a local coordinate at $p$.
Then there is an open subvariety $U$ of $X$ containing $p$ such that $f$ gives an \'etale morphism $U \to \mathbb{A}^1$.
So $f(U(K))$ is a nonempty finite \'etale image in $\Aa^1(K)$.

\meno
$(2) \Rightarrow (3)$ is immediate.

\medskip
$(3) \Rightarrow (4)$.
Suppose that $V$ is a $K$-variety and $V(K)$ is infinite and $\cE_K$-discrete.
Let $W$ be the Zariski closure of $V(K)$ in $V$.
As $V(K)$ is infinite, $\dim W \ge 1$.
Let $W^*$ be a positive-dimensional irreducible component of $W$.
Then $W^*(K)$ is Zariski dense in $W^*$ by Lemma~\ref{K-point dense}.
Let $O$ be the smooth locus of $W^*$. 
Then $O$ is an open subvariety of $W^*$.
By Lemma~\ref{smooth-dimension}, $O$ is nonempty, so $O(K)$ is nonempty, and any $p \in O(K)$ is $\cE_K$-isolated in $W^*(K)$.



\meno
$(4) \Rightarrow (1)$.
Suppose that $p \in V(K)$ is $\cE_K$-isolated and smooth.
Let $O$ be the smooth locus of $V$.
Then $p$ is $\cE_K$-isolated in $O(K)$ and $O$ is a non-empty open subvariety of $V$. By irreducibility of $V$, one has $\dim O=\dim V$.
Let $X$ be a $K$-variety and $f : X \to O$ be an \'etale morphism such that $f(X(K)) = \{p\}$.
Then $\dim X = \dim O = \dim V \ge 1$, $X$ is smooth as $O$ is smooth, and $X(K)$ is finite as $X(K) \to O(K)$ is finite-to-one.
Apply Fact~\ref{fact:large}.
\end{proof}

\subsection{Hausdorffness}

\begin{theorem}
\label{thm:haus}
The following are equivalent:
\begin{enumerate}
\item $K$ is not separably closed,
\item The $\cE_K$-topology on $V(K)$ is Hausdorff for any quasi-projective $K$-variety $V$.
\end{enumerate}
\end{theorem}

We first gather a few lemmas.

\begin{lemma}
\label{lem:hausdorff-restrict}
Suppose $L/K$ is finite and the $\cE_L$-topology on $\Aa^1_L(L)$ is Hausdorff.
Then the $\cE_K$-topology on $\Aa^1(K)$ is Hausdorff.
\end{lemma}

\begin{proof}
The $ \rtp(\cE_L)$-topology on $K$ is Hasudorff as a subspace of a Hausdorff space is Hausdorff.
By Theorem~\ref{thm:general}, $\cE_K$ refines $ \rtp(\cE_L)$, so the $\cE_K$-topology on $K$ is Hausdorff.
\end{proof}

We can now prove Theorem~\ref{thm:haus}, following part of the proof of Theorem D.

\begin{proof}
Proposition~\ref{prop:zar} shows that $(2)$ implies $(1)$.
We show that $(1)$ implies $(2)$.
Suppose that $K$ is not separably closed.
By Proposition~\ref{prop:hd}  and Lemma~\ref{lem:hausdorff-restrict} it suffices to produce a finite extension $L$ of $K$ and a disjoint pair of nonempty $\cE_L$-open subsets of $\Aa^1_L(L)$.
By Lemma~\ref{lemma:galois} there is a finite extension $L$ of $K$ such that either
\begin{enumerate}
\item the $p$th power map $L^\times \to L^\times$ is not surjective for some prime $p \ne \Chara(K)$,
\item or the Artin-Schreier map $L \to L$ is not surjective.
\end{enumerate}
In the first case we fix $p$ and let $P$ be the image of the $p$th power map $\Gg_m(L) \to \Gg_m(L)$ and in the second case we let $P$ be the image of the Artin-Schreier map $\Aa^1_L(L) \to \Aa^1_L(L)$.
In the first case $P$ is a non-trivial $\cE_L$-open subgroup of $\Gg_m(L)$ and in the second case $P$ is a non-trivial $\cE_L$-open subgroup of the additive group of $L$.
Fix $\alpha \in \Gg_m(L) \setminus P$.
In the first case let $P^* = \alpha P$ and in the second case let $P^* = \alpha + P$.
In either case $P,P^*$ are nonempty disjoint $\cE_L$-open subsets of $\Aa^1_L(L)$.
\end{proof}

Corollary~\ref{cor:separable} below follows from Proposition~\ref{prop:zar}  and Theorem~\ref{thm:haus}.

\begin{corollary}
\label{cor:separable}
The \'etale-open topology over $K$ agrees with the Zariski topology if and only if $K$ is separably closed.
\end{corollary}







\subsection{V-topologies}
\label{section:thmb-redux}

We complete the proof of Theorem B.  We first prove a lemma.

\begin{lemma}
\label{lem:hensel-etale}
Fix $n \geq 2$.
Then there is an \'etale image  $U \subseteq \Aa^{n+1}(K)$ containing $(0,\ldots,0)$ such that $x^{n + 2} + x^{n + 1} + \alpha_{n} x^{n} + \cdots + \alpha_1 x + \alpha_0$ has a root in $K$ for any $(\alpha_0,\ldots,\alpha_{n}) \in U$.
\end{lemma}

\begin{proof}
Let $V$ be $\Spec K[y_0,\ldots,y_{n},x]/(x^{n + 2} + x^{n + 1} + y_{n} x^{n} + \cdots + y_1 x + y_0)$.
So $V$ is an affine subvariety of $\Aa^{n + 2}(K)$.
Let $\pi : V \to \Aa^{n + 1}$ be the projection onto the first $n + 1$ coordinates. We claim that $\pi$ is \'etale at $(0,...,0,-1)$. As $V$ has dimension $n+1$, it suffices to check that $(0,...,0,-1)$ is a regular point and $\pi$ induces a surjective map between the tangent spaces. The tangent space of $V$ at $(0,...,0,-1)$ is given by 
$$\sum^n_{i=0} (-1)^iy_i+((n+2)(-1)^{n+1}+(n+1)(-1)^n)(x+1)=0.$$  Simplified, we get $(-1)^{n + 1}(x+1) + \sum^n_{i=0} (-1)^iy_i = 0$. 
So the induced map on the tangent space is an isomorphism as the coefficient of $x + 1$ is non-zero.
So there is an open subvariety $W$ of $V$ containing $(0,\ldots,0,-1)$ on which $\pi$ is \'etale.
Let $U$ be $\pi(W(K))$. 
\end{proof}

\begin{theorem}
\label{thm:b-2}
Suppose that there is a V-topology $\uptau$ on $K$ that induces the \'etale-open system of topologies $\cE_K$.
Then $\uptau$ is t-Henselian and $K$ is not separably closed.
\end{theorem}
\begin{proof}
Note that $\cE_K$ is not the Zariski topology, and hence $K$ is not separably closed by Proposition~\ref{prop:zar}.
Fix $n$.
As the $\cT_\uptau$-topology on $\Aa^n(K)$ is the product topology, Lemma~\ref{lem:hensel-etale} implies that there is a $\uptau$-open neighborhood $U$ of $0$ such that if $\alpha_0,\ldots,\alpha_n \in U$ then the polynomial $x^{n + 2} + x^{n + 1} + \alpha_n x^n + \cdots + \alpha_1 x + \alpha_0$ has a root in $K$.
Hence $\uptau$ is t-Henselian.
\end{proof}

Theorems~\ref{thm:hensel-main} and \ref{thm:b-2} are the two claims of Theorem B.

\begin{theorem}
\label{thm: hensel-main2}
The following are equivalent:
\begin{enumerate}
\item $K$ is non-separably closed and t-Henselian.
\item There is a t-Henselian topology $\uptau$ on $K$ that induces $\cE_K$.
\item There is a V-topology $\uptau$ on $K$ that induces $\cE_K$.
\end{enumerate}
\end{theorem}

\begin{proof}
Theorem~\ref{thm:hensel-main} shows that $(1)$ implies $(2)$.  T-Henselian topologies are V-topologies by definition, so $(2)$ implies $(3)$.  Lastly, $(3)$ implies $(1)$ by Theorem~\ref{thm:b-2}.
\end{proof}

\begin{remark}
\label{remark:field-topology}
Theorem~\ref{thm: hensel-main2} characterizes $K$ such that $\cE_K$ is induced by a V-topology.
It is an open question to characterize $K$ such that $\cE_K$ is induced by a field topology.
Suppose $R$ is a Henselian regular local ring of dimension $\ge 2$  and $L$ is the fraction field of $R$.
Pop~\cite{Pop-henselian} shows that $L$ is large.
Furthermore $\{ \alpha R + \beta : \alpha \in L^\times, \beta \in L \}$ is a basis for a field topology on $L$ which is not a V-topology.
In forthcoming work we will show that $\cE_L$ is induced by this field topology.
This covers $R = F[[x_1,\ldots,x_n]]$ for an arbitrary field $F$ and $n \ge 2$.
In Section~\ref{section:pseudofinite} we show that if $K$ is either pseudofinite of odd characteristic or an infinite non-quadratically closed algebraic extension of an odd characteristic finite field then $\Sa E_K$ is not induced by a field topology.
We believe, more generally, that if $K$ is pseudo real closed and not real closed then $\Sa E_K$ is not induced by a field topology.
In particular, if $K$ is $\mathrm{PAC}$ then we believe that $\Sa E_K$ is not given by a field topology.
Likewise, we believe that if $K$ is pseudo $p$-adically closed and not $p$-adically closed, then $\Sa E_K$ is not induced by a field topology.
\end{remark}

We conclude this section with two corollaries to Theorem B.

\begin{corollary}
\label{cor:ordered}
Suppose that $<$ is a field order on $K$.
Then the \'etale-open topology over $K$ is induced by $<$ if and only if the $<$-topology on $K$ is t-Henselian.
\end{corollary}

Corollary~\ref{cor:ordered} follows from Theorem~\ref{thm:hensel-main}, Theorem~\ref{thm:b-2}, the fact that the $<$-topology is a V-topology, and the observation that an ordered field is not separably closed.
Examples of non real closed t-Henselian ordered fields include $F((t))$ for an arbitrary ordered field $F$.

\begin{corollary}
\label{cor:saturated}
Suppose that $K$ is $\aleph_1$-saturated.
The following are equivalent:
\begin{enumerate}
\item $K$ is non-separably closed and Henselian,
\item $\Sa E_K$ is induced by a non-trivial Henselian valuation on $K$,
\item $\Sa E_K$ is induced by a non-trivial valuation on $K$.
\end{enumerate}
\end{corollary}

Corollary~\ref{cor:saturated} follows from Theorem~\ref{thm: hensel-main2} and the fact that any t-Henselian topology on an $\aleph_1$-saturated field is induced by a Henselian valuation~\cite{Prestel1978}.

\subsection{Connectedness}
\label{section:connected}
It is a central idea in real algebraic geometry that an ordered field $(L,<)$ is real closed if and only if the $<$-topology is ``connected" with respect to polynomial functions.
We give a similar characterization in terms of the \'etale-open topology.
We show that $K$ is neither separably closed nor real closed if and only if there is a non-trivial $\cE_K$-clopen \'etale image in $\Aa^1(K)$.
As a corollary we show that the \'etale-open topology on $\Aa^1(K)$ is connected if and only if $K$ is either separably closed or isomorphic to $\Rr$.

\meno
Suppose $V$ is a $K$-variety.
A \textbf{clopen \'etale image} is an $\cE_K$-closed \'etale image in $V(K)$.
We say that $V(K)$ is \textbf{\'etale connected} if the only clopen \'etale images in $V(K)$ are $\emptyset$ and $V(K)$, and $V(K)$ is \textbf{\'etale disconnected} otherwise. 
We say $V(K)$ is \textbf{\'etale totally separated} if for any $\alpha,\beta \in V(K), \alpha \ne \beta$ there is a clopen \'etale image $U$ in $V(K)$ such that $\alpha \in U, \beta \notin U$.

\begin{theorem}
\label{thm:connected}
Suppose $K$ is not separably closed.
The following are equivalent:
\begin{enumerate}
\item $\Aa^1(K)$ is not \'etale connected,
\item $\Aa^1(K)$ is \'etale totally separated,
\item $V(K)$ is \'etale totally separated for any quasi-projective $K$-variety $V$,
\item $K$ is not real closed.
\end{enumerate}
Furthermore the following are equivalent:
\begin{enumerate}
\setcounter{enumi}{4}
\item the $\cE_K$-topology on $\Aa^1(K)$ is not connected,
\item the $\cE_K$-topology on $\Aa^1(K)$ is totally separated,
\item $(V(K),\cE_K)$ is totally separated for any quasi-projective $K$-variety $V$,
\item $K$ is not isomorphic to $\Rr$.
\end{enumerate}
\end{theorem}

We first gather some  lemmas.

\begin{lemma}
\label{lem:extension-disconnect}
Let $L$ be a finite extension of $K$.
If $\Aa^1_L(L)$ is \'etale disconnected then $\Aa^1(K)$ is \'etale disconnected.
\end{lemma}

\begin{proof}
Suppose that $U$ is a nontrivial clopen $L$-\'etale image in $\Aa^1_L(L)$.
After applying an affine transformation $\Aa_L^1 \to \Aa_L^1$ we suppose that $0 \in U, 1 \notin U$.
By Theorem~\ref{thm:general}, $U \cap \Aa^1(K)$ is an $\cE_K$-clopen subset of $\Aa^1(K)$.
By Lemma~\ref{lem:basic-down}, $U \cap \Aa^1(K)$ is a $K$-\'etale image.
\end{proof}

\begin{lemma}
\label{lem:powers-0}
Suppose $n$ is prime to $\mathrm{Char}(K)$.
Then the set of nonzero $n$th powers is a clopen \'etale image in $\Gg_m(K)$.
If $K$ has positive characteristic then the image of the Artin-Schreier map $\Aa^1(K) \to \Aa^1(K)$ is a clopen \'etale image in $\Aa^1(K)$.
\end{lemma}

Lemma~\ref{lem:powers-0} follows by the same argument as in the case of a field topology and the fact that the $n$th power map is \'etale when $n$ is prime to $\Chara(K)$.
We omit the proof.

\medskip

Lemma~\ref{lem:power-frontier} below will be used to analyze the image of the $n$th power map.
The \textbf{frontier} of a subset $X$ of a topological space $S$ is the set of closure points $p \in S$ of $X$ such that $p\notin X$.

\begin{lemma}
\label{lem:power-frontier}
Suppose that $\cT$ is a non-discrete system of topologies over $K$.
Fix $n$ and let $P$ be the set of nonzero $n$th powers.
Then zero is a $\cT$-frontier point of $P$.
\end{lemma}

\begin{proof}
By Proposition~\ref{prop:discrete}, the $\cT$-topology on $\Aa^1(K)$ is non-discrete.
As the $\cT$-topology is invariant under affine transformations, no point of $\Aa^1(K)$ is isolated.  
Let $U$ be a $\cT$-open neighborhood of zero and $f:\Aa^1(K) \to \Aa^1(K)$ be given by $f(a) = a^n$.
We show that $P \cap U \neq \emptyset$.
As $f$ is continuous there is a $\cT$-open neighborhood $V$ of zero such that $f(V) \subseteq U$.
By non-discreteness $V \setminus \{0\}$ is nonempty.
As $f(V \setminus \{0\}) \subseteq P$ we have $P \cap U \neq \emptyset$.
\end{proof}

We now prove Theorem~\ref{thm:connected}.

\begin{proof}
It is clear that $(2)$ implies $(1)$ and $(1)$ implies $(2)$ as the action of the affine group on $\Aa^1(K)$ is two-transitive and an affine image of an \'etale image in $\Aa^1(K)$ is an \'etale image.
The implication $(3) \Rightarrow (2)$ is immediate and the implication $(2) \Rightarrow (3)$ may be proven by following the proof of Proposition~\ref{prop:tot-dis} and applying Lemma~\ref{lem:cts} when necesssary.

\meno
It remains to show that $(1)$ and $(4)$ are equivalent.
We show that $(1)$ implies $(4)$.
Suppose $K$ is real closed.
By Corollary~\ref{rcf-case} the \'etale-open topology agrees with the order topology on $\Aa^1(K)$.
By the Tarski-Seidenberg theorem every \'etale image in $\Aa^1(K)$ is semialgebraic and is hence a finite union of intervals.
It is easy to see that a clopen finite union of intervals must either be $\emptyset$ or $\Aa^1(K)$.

\meno
We show that $(4)$ implies $(1)$.

\begin{claim-star}
\label{lem:minus-one}
Suppose $p \neq \mathrm{Char}(K)$ is prime, the $p$th power map $\Gg_m(K) \to \Gg_m(K)$ is not surjective, and $-1$ is a $p$th power.
Then $\Aa^1(K)$ is \'etale disconnected.
\end{claim-star}

\begin{claimproof}
We work in the \'etale-open topology.
If $\Aa^1(K)$ is discrete then it is disconnected, so we suppose that $\Aa^1(K)$ is not discrete.
Let $P$ be the set of nonzero $p$th powers.
By Lemma~\ref{lem:powers-0}, $P$ is an \'etale image in $\Aa^1(K)$.
By Lemma~\ref{lem:intersection}, $P \cup (1 + P)$ is an \'etale image in $\Aa^1(K)$.
We show that $P \cup (1 + P)$ is a nontrivial clopen subset of $\Aa^1(K)$.
It suffices to show that $P \cup (1 + P)$ is closed and $P \cup (1 + P) \neq \Aa^1(K)$.
We first show that $P \cup (1 + P)$ is closed.
By Lemma~\ref{lem:powers-0}, $P$ is clopen in $\Gg_m(K)$, and so $0$ is the only possible frontier point of $P$.  Then $1$ is the only possible frontier point of $1 + P$.
Suppose that $\alpha \in \Aa^1(K)$ is a frontier point of $P \cup (1 + P)$.
Observe that $\alpha \notin P$,  $\alpha \notin 1 + P$, and $\alpha$ is either a limit point of $P$ or $1 + P$.
So either $\alpha = 0$ or $\alpha = 1$.
This is a contradiction as $1 \in P$ and $0 \in 1 + P$.
We now show that $P \cup (1 + P) \neq \Aa^1(K)$.
Fix $\beta \in \Gg_m(K) \setminus P$.
As $P$ is clopen in $\Gg_m(K)$ there is an open neighborhood $W$ of $\beta$ such that $W \cap P \neq \emptyset$.
Then $W - \beta$ is an open neighborhood of $0$ so we may fix $\gamma \in P \cap (W - \beta)$, by Lemma~\ref{lem:power-frontier}.
Note that $\beta + \gamma \in W$.
As $-1 \in P$ we have $-1/\gamma \in P$.
As $-1/\gamma \in P$ and $\beta, \beta + \gamma \notin P$ we have $-
\beta/\gamma \notin P$ and $-(\beta/\gamma) - 1 = -(\beta + \gamma)/\gamma \notin P$.
Therefore $-\beta/\gamma \notin P \cup (1 + P)$.
\end{claimproof}

Suppose that $K$ is not real closed and let $\imag^2 = -1$.
By Fact~\ref{fact:artin}, $K[\imag]$ is neither separably closed nor real closed.
By Lemma~\ref{lemma:galois} there is a finite extension $L$ of $K[\imag]$ such that either
\begin{enumerate}
\item the Artin-Schreier map $L \to L$ is not surjective, or
\item there is a prime $p \neq \mathrm{Char}(L)$ such that the $p$th power map $L^\times \to L^\times$ is not surjective.
\end{enumerate}
In the first case $\Aa^1_L(L)$ is \'etale disconnected by Lemma~\ref{lem:powers-0}.
In the second case an application of the claim shows that $\Aa^1_L(L)$ is \'etale disconnected.
Apply Lemma~\ref{lem:extension-disconnect}.

\meno
Proposition~\ref{prop:tot-dis} shows the equivalence of $(5), (6)$, $(7)$.
The equivalence of $(5)$ and $(8)$ follows by the argument above and the fact that any connected ordered field is isomorphic to $\Rr$.
\end{proof}

We now record two model-theoretic corollaries.
Recall that if $V$ is a $K$-variety and $\uptau$ is a topology $\uptau$ on $V(K)$ then $V(K)$ \textbf{definably connected} if there are no non-trivial definable $\uptau$-clopen subsets of $V(K)$ and is \textbf{definably totally separated} if for any distinct $a,b \in V(K)$ there is a definable clopen subset $U$ of $V(K)$ such that $a \in U, b \notin U$.

\begin{corollary}
\label{cor:connected-1}
The following are equivalent:
\begin{enumerate}
\item $\Aa^1(K)$ is not definably connected,
\item $\Aa^1(K)$ is definably totally separated,
\item $V(K)$ is definably totally separated for any quasi-projective $K$-variety $V$,
\item $K$ is neither separably closed nor real closed.
\end{enumerate}
\end{corollary}

Corollary~\ref{cor:connected-1}
follows by Theorem~\ref{thm:connected} and the well-known fact that the order topology on a real closed field is definably connected.

\begin{corollary}
\label{cor:exis-connected}
If there is a non-trivial definable $\cE_K$-clopen subset of $\Aa^1(K)$ then there is a non-trivial existentially definable $\cE_K$-clopen subset of $\Aa^1(K)$.
\end{corollary}

Corollary~\ref{cor:exis-connected} follows from Theorem~\ref{thm:connected}, Corollary~\ref{cor:connected-1}, and the fact that any $K$-\'etale image is existentially $K$-definable.  Corollary~\ref{cor:exis-connected} is similar in spirit to Corollary~\ref{cor:existential} and Section~\ref{section:podewski} below.

\subsection{Local compactness}
\label{section:loc-compact}
If $K$ is separably closed then $\cE_K$ agrees with the Zariski topology, hence the $\cE_K$-topology on the $K$-points of any $K$-variety is compact.

\begin{theorem}
\label{thm:loc-compact}
Suppose that $K$ is not separably closed.
The following are equivalent:
\begin{enumerate}
\item there is a $K$-variety $V$ and an infinite $\cE_K$-open subset $U$ of $V(K)$ such that $(U,\cE_K)$ is locally compact,
\item The \'etale-open topology on $\Aa^1(K)$ is locally compact,
\item $K$ is a local field, i.e., $K$ admits a locally compact field topology.
\end{enumerate}
\end{theorem}

We gather some necessary facts.

\begin{fact}
\label{fact:semi-top}
Suppose $G$ is a group and $\uptau$ is a topology on $G$ such the map $G \to G$, $b \mapsto b^{-1}$ is $\uptau$-continuous and the map $G \to G$ given by $a \mapsto ba$ and $a \mapsto ab$ is $\uptau$-continuous for any $b \in G$.
If $\uptau$ is locally compact Hausdorff then $\uptau$ is a group topology.
\end{fact}

\begin{fact}
\label{fact:lclh}
Suppose that $S$ is a topological space.
If $S$ is locally Hausdorff and locally compact then every open subset of $S$ is locally compact.
\end{fact}

Fact~\ref{fact:semi-top} is a theorem of Ellis~\cite{ellis-para}.
Fact~\ref{fact:lclh} is a generalization of a familiar fact about Hausdorff locally compact  spaces. See \cite[Proposition~3.6]{niefield} for a proof.

\begin{lemma}
\label{lem:project}
Suppose that $\cT$ is a system of topologies over $K$ and $V,W$ are $K$-varieties.
Then the coordinate projection $\pi :(V \times W)(K) \to V(K)$ is $\cT$-open.
\end{lemma}

\begin{proof}
Suppose that $U$ is a $\Sa T$-open subset of $(V \times W)(K)$.
For each $\beta \in W(K)$ we let $U_\beta = U \cap [ V(K) \times \{\beta\}]$.
We have 
$\pi(U) = \bigcup_{\beta \in W(K)} \pi(U_\beta )$ so it suffices to fix $\beta \in W(K)$ and show that $\pi(U_\beta)$ is $\Sa T$-open.
Note that $V \times \{\beta\}$ is a closed subvariety of $V \times W$ so $U_\beta$ is $\Sa T$-open in $V(K) \times \{\beta\}$.
The projection $V \times \{\beta\} \to V$ is an isomorphism so the projection $V(K) \times \{\beta\} \to V(K)$ is a $\Sa T$-homeomorphism.
So each $\pi(U_\beta)$ is $\Sa T$-open.
\end{proof}

We now prove Theorem~\ref{thm:loc-compact}.

\begin{proof}
$(2) \Rightarrow (1)$ is clear.

\medskip
$(1) \Rightarrow (2)$.  
Suppose that $V$ is a $K$-variety and $U$ is an infinite locally compact open subset of $V(K)$.
We claim that $\Aa^1(K)$ is locally compact.
As the action of the affine group on $\Aa^1(K)$ is transitive it is enough to produce a nonempty open subset of $\Aa^1(K)$ with compact closure.
Note that $V(K)$ is locally Hausdorff by Theorem~\ref{thm:haus}.
After possibly replacing $V$ with the Zariski closure of $U$ in $V$ (see Remark~\ref{z-closure}) we may suppose that $U$ is Zariski dense in $V$.
As $U$ is infinite $V$ has dimension $\ge 1$.
By Lemma~\ref{smooth-dimension}, there is a non-empty open subvariety $O \subseteq V$ such that $O$ is smooth of dimension $n = \dim(V)$.
Fix $p \in O(K) \cap U$.
Let $f_1, f_2, \ldots, f_n \in \mathcal{O}_p$ be local coordinates at $p$.
There is an open subvariety $W$ of $O$ such that $\vec{f} = (f_1,\ldots,f_n)$ gives an \'etale morphism $W \to \Aa^n$.
Let $\pi$ be a coordinate projection $\Aa^n \to \Aa^1$ and $g : W \to \Aa^1$ be $g = \pi \circ \vec{f}$.
By Lemma~\ref{lem:project}, $g$ gives an open map $W(K) \to \Aa^1(K)$.
By Theorem~\ref{thm:haus} any affine open subset of $V$ is Hausdorff, so $U$ is locally Hausdorff.
By Fact~\ref{fact:lclh} $U \cap W(K)$ is locally compact.
After possibly shrinking $W$ we may suppose that $W$ is an affine open subvariety of $V$, so $U \cap W(K)$ is locally compact Hausdorff.  Therefore there is a nonempty open subset $U'$  of $U \cap W(K)$ such that the closure $Z$ of $U'$ in $U \cap W(K)$ is compact.  Then $g(U')$ is open as $g : W(K) \to \Aa^1(K)$ is an open map.  The open set $g(U')$ is contained in the compact set $g(Z)$, so $g(U')$ has compact closure.

\meno
$(2) \Rightarrow (3)$.  Assume $(2)$.
By Theorem~\ref{thm:haus} the \'etale-open topology on $\Aa^1(K)$ is Hausdorff.  For $\alpha \in K$, the maps $x \mapsto \alpha x$, $x \mapsto \alpha + x$, and $x \mapsto -x$ are continuous, because these maps come from variety morphisms $\Aa^1 \to \Aa^1$.  Likewise, the map $x \mapsto 1/x$ is continuous on $K^\times$ as it comes from a morphism $\Gg_m \to \Gg_m$.
By Fact~\ref{fact:semi-top} the \'etale-open topology on $\Aa^1(K)$ is a field topology, so $K$ is a local field.

\meno
$(3) \Rightarrow (2)$.  Local fields are t-Henselian; apply Theorem~\ref{thm:hensel-main}.
\end{proof}

\section{Another example where the \'etale-open topology is not a field topology}
\label{section:pseudofinite}
In this section, we give an example of a field $K$ for which the \'etale-open system $\cE_K$ is not induced by a field topology---and in fact, the \'etale-open topology $\cE_{\Aa^1}$ on $K = \Aa^1(K)$ is not even a field topology.
We already have one example: separably closed fields (Proposition~\ref{prop:zar}).  However, separably closed fields are in many ways exceptional (see Theorems~\ref{thm:haus} and \ref{thm: hensel-main2}), and so it is good to have another example.

\meno
Throughout this section $p$ is a fixed odd prime and $\palg$ is an algebraic closure of $\Ff_p$.
We take all algebraic extensions of $\Ff_p$ to be subfields of $\palg$.

\begin{proposition}
\label{prop:finite}
Suppose $K$ is an infinite non-quadratically closed algebraic extension~of~$\Ff_p$.
Then $\{ (\alpha,\beta) \in \Aa^2(K) : \exists c \in K^\times ( \alpha - \beta = c^2) \}$ is not open in the product topology on $\Aa^2(K)$ associated to the $\cE_K$-topology on $\Aa^1(K)$.
Hence the $\cE_K$-topology on $K$ is not a field topology.
\end{proposition}

Note that the $\cE_K$-topology on $K$ is Hausdorff as $K$ is not separably closed.
The Hasse-Weil estimates imply that any infinite subfield of $\palg$ is $\mathrm{PAC}$ (pseudo-algebraically closed), \cite[11.2.4]{field-arithmetic}, and any $\mathrm{PAC}$ field is large \cite{Pop-little}.
As previously mentioned, we believe, more generally,  that $\cE_K$ is not induced by a field topology when $K$ is $\mathrm{PAC}$. 

\meno
We will make use of the formalism of Steinitz numbers; see~\cite{BS-Fp} for details and proofs.
Let $\mathbb{P}$ be the set of prime numbers.
A \textbf{Steinitz number} is a formal product $s = \prod_{q \in \mathbb{P}} q^{e(q)}$ for some function $e : \mathbb{P} \to \Nn \cup \{\infty\}$.
We declare $\val_q(s) = e(q)$ for all $q \in \mathbb{P}$.
Suppose that $s,t$ are Steinitz numbers.
If $\val_q(s) < \infty$ for all $q \in \mathbb{P}$ and $\val_q(s) = 0$ for all but finitely many $q \in \mathbb{P}$ then we identify $s$ with a natural number in the natural way.
We multiply Steinitz numbers in the natural way and we say that $s$ divides $t$ if $\val_q(s) \le \val_q(t)$ for all $q \in \mathbb{P}$.
We declare $\Ff_{p^s}$ to be the union of all $\Ff_{p^n}$, $n$ dividing $s$.
Every subfield of $\palg$ is of the form $\Ff_{p^s}$ for a unique Steinitz number $s$ and $\Ff_{p^t}$ is a degree $d$ extension of $\Ff_{p^s}$ if and only if $t = ds$.

\meno
\textbf{For the remainder of this section $K$ is an infinite non-quadratically closed subfield of $\palg$.}
Let $s$ be the Steinitz number such that $K = \Ff_{p^s}$.
As $K$ is not quadratically closed $\val_2(s) < \infty$.
Let $K_0$ be $\Ff_{p^m}$ where $m = 2^{\val_2(s)}$.
So $K_0$ is finite.

\begin{lemma}
\label{lem:intermediate}
Suppose that $K_0 \subseteq L \subseteq K$ is a field and $\alpha \in L$.
Then $\alpha$ is a square in $K$ if and only if $\alpha$ is a square in $L$.
\end{lemma}

\begin{proof}
The left to right implication is obvious; we prove the other implication.
Let $L = \Ff_{p^t}$.
We have $\val_2(t) = \val_2(s)$.
Suppose that $\alpha$ is not a square in $L$ and fix $\beta \in \palg$ such that $\beta^2 = \alpha$.
As $p \ne 2$, $\beta$ has degree $2$ over $L$ so $L(\beta) = \Ff_{p^{2t}}$.
So $\val_2(2t) = 1 + \val_2(t) > \val_2(s)$.
Hence $2t$ does not divide $s$, hence $L(\beta)$ is not a subfield of $K$.
So $\alpha$ is not a square in $K$.
\end{proof}

\meno
We now apply the Hasse-Weil bounds to prove two facts about finite fields.
Neither is original.
Fact~\ref{paley-ish} below follows from the combinatorics of Paley graphs and tournaments; see for example \cite{CGW,CG-tournament}.

\begin{fact}
\label{paley-ish}
Suppose that $F$ is a finite extension of $K_0$ and $\beta_1,\ldots,\beta_k \in F$ are pairwise distinct.
Let $S$ be the set of $\alpha \in F$ such that $\alpha - \beta_i$ is a non-zero square in $F$ for all $i$.
If $|F|$ is sufficiently large then $|S| < 2^{1-k}|F|$.
\end{fact}

\begin{proof}
Let $C$ be the quasi-affine $F$-curve given by the equations
  \begin{align*}
      x - \beta_i &= y_i^2 \\
      x - \beta_i & \ne 0
  \end{align*}
for $1 \le i \le k$.
Then $C$ is geometrically irreducible, as $\palg[x,y_1,\ldots,y_k]/(x - \beta_i - y_i^2 : 1 \le i \le k)$ is a domain.\footnote{If $R$ is a unique factorization domain, and $p_1, \ldots, p_k$ are pairwise non-equivalent primes in $R$, then $R[y_1,\ldots,y_k]/(y_i^2 - p_i : i \in \{1, \ldots, k\})$ is a domain.
To see this, let $F = \Frac(R)$, and use Galois theory to see that $F(\sqrt{p_1},\ldots,\sqrt{p_k})$ has degree $2^k$ over $F$.
Therefore $F[y_1,\ldots,y_k]/(y_i^2 - p_i : i \in \{1, \ldots, k\}))$ is isomorphic as an $F$-algebra to  $F(\sqrt{p_1},\ldots,\sqrt{p_k})$.
By inspection, the ring $S := R[y_1,\ldots,y_k]/(y_i^2 - p_i : i \in \{1 \ldots, k\}))$ embeds into $F[y_1,\ldots,y_k]/(y_i^2 - p_i : i \in \{1, \ldots, k\}))$, and so $S$ is a domain.
In our case, $R$ is the UFD $\palg[x]$ and $p_i$ is the prime $x - \beta_i$ for each $i$. } 
The Hasse-Weil bounds yield $|C(F)| < 2|F|$ when $|F|$ is sufficiently large.  
The set $S$ is the image of $C(F)$ under the projection to the first coordinate.
Because $p \ne 2$, the fibers of $C(F) \to S$ have cardinality exactly $2^k$.
\end{proof}

The following fact is a special case of the results of \cite{CDM-finite}.

\begin{fact}
\label{fact:cdm}
Let $F$ be a finite extension of $\Ff_p$, $V$ be an $F$-variety with $V(F)$ non-empty, and $f : V \to \Aa^1_F$ be an \'etale morphism.
Then there is $\varepsilon > 0$ such that if $E$ is a finite extension of $F$ and $|E|$ is sufficiently large then $|f_E(V_E(E))| \ge \varepsilon |E|$.
\end{fact}

\begin{proof}
As $f$ is \'etale, $V$ is a smooth curve.
As $V(F)$ is nonempty, some connected component of $V$ is a geometrically irreducible smooth curve.\footnote{Let $W$ be one of the connected components of $V$ containing an $F$-point $q$.  The base change $W_{\palg}$ is smooth, so its irreducible components are its connected components.  Let $X$ be the connected component of $W_{\palg}$ containing $q$, and let $Y$ be the union of the remaining connected components (possibly empty).  The Galois action $\Gal(\palg/F)$ fixes $q$, and therefore fixes $X$ and $Y$ setwise.  Therefore $X$ and $Y$ descend to varieties $X^*$ and $Y^*$ over $F$.
Then $W$ is the disjoint union of $X^*$ and $Y^*$, so $W = X^*$ and $Y^* = \emptyset$ because $W$ is connected.  Finally, $X^*$ is geometrically irreducible because $X$ is irreducible.}
The Hasse-Weil bounds imply $|V_E(E)| \ge (1/2)|E|$ when $|E|$ is sufficiently large.
As $f$ is \'etale there is $k$ such that $f_E$ is $k$-to-one for any finite field $E$ extending $F$.
So if $|E|$ is sufficiently large then $|f_E(V_E(E))| \ge (1/2k)|E|$.
\end{proof}

We now prove Proposition~\ref{prop:finite}.

\begin{proof}
Given a field $L$ we let $E_L$ be the set of $(a,b) \in L^2$ such that $a - b$ is a non-zero square.
We show that $E_K$ is not $\Sa P$-open.  Suppose otherwise.  As $E_K \ne \emptyset$, there are non-empty \'etale images $U_0, U_1$ in $\Aa^1(K)$ such that $U_0 \times U_1 \subseteq E_K$.
Given $i \in \{0,1\}$ let $X_i$ be a $K$-variety and $f_i : X_i \to \Aa^1$ be an \'etale morphism such that $U_i = f_i(V_i(K))$.
Let $K_1$ be a finite subfield of $K$ such that $K_1$ contains $K_0$, each $V_i$ and $f_i$ is defined over $K_1$, and each $V_i(K_1)$ is nonempty.
For each field $K_1 \subseteq L \subseteq K$ and $i \in \{0,1\}$ we let $U_i(L) = f_i(V_i(L))$.
Lemma~\ref{lem:intermediate} shows that if $K_1 \subseteq L \subseteq K$ is a field then $E_L = E_K \cap \Aa^2(L)$, hence $U_0(L) \times U_1(L)\subseteq E_L$.
Applying Fact~\ref{fact:cdm} we obtain $\varepsilon > 0$ such that if $K_1 \subseteq L \subseteq K$ is finite and $|L| \ge n$ then $|U_0(L)|,|U_1(L)| \ge \varepsilon |L|$.
Fix $k$ such that $2^{1-k} < \varepsilon$.  Take $K_2$ finite with $K_1 \subseteq K_2 \subseteq K$ and $|K_2| \ge n$ and $|U_1(K_2)| \ge k$.  Take distinct $b_1, \ldots, b_k \in U_1(K_2)$.  Suppose $L$ is finite and $K_2 \subseteq L \subseteq K$.  If $a \in U_0(L)$, then $(a,b_i) \in E_L$ for $i \in \{ 1, \ldots, n\}$.
By Fact~\ref{paley-ish}, $|U_0(L)| \le 2^{1-k}|L| < \varepsilon|L|$ when $|L|$ is sufficiently large.  This contradicts the choice of $\varepsilon$.

\meno
The set of non-zero squares in $\Aa^1(K)$ is $\cE_K$-open, because $p \ne 2$.  
Hence the subtraction map
\begin{equation*}
    (\Aa^1(K),\cE_K) \times (\Aa^1(K),\cE_K) = (\Aa^2(K),\Sa P) \to (\Aa^1(K),\cE_K)
\end{equation*}
is not continuous, and the $\cE_K$-topology on $\Aa^1(K)$ is not a field topology.
\end{proof}

\subsection{Pseudofinite fields}
\label{section:pseudo-finite}
We assume some familiarity with pseudofinite fields.

\begin{proposition}
\label{prop:pseudofinite}
Suppose that $K$ is pseudofinite of odd characteristic.
Then the $\cE_K$-topology on $K$ is not a field topology.
\end{proposition}

It is worth recalling that if $s$ is a Steinitz number then $\Ff_{p^s}$ is pseudofinite if and only if $\val_q(s) < \infty$ for all primes $q$ and $s$ is not a natural number\footnote{This follows by Ax's theorem~\cite{Ax-fin} that a field is pseudofinite if and only if it is perfect, $\mathrm{PAC}$, and admits a unique extension of each degree and the fact that infinite algebraic extensions of finite fields are $\mathrm{PAC}$.}.
Hence there is some overlap between Propositions \ref{prop:finite} and \ref{prop:pseudofinite} but latter does not imply the former.

\meno
The proof is very similar to that of Proposition~\ref{prop:finite} so we only give a sketch.
We only treat the case when $K$ is the ultraproduct of a sequence $(K_i : i \in \Nn)$ of finite fields with respect to a nonprincipal ultrafilter $\upomega$ on $\Nn$; our argument extends to the general case via routine but somewhat tedious model-theoretic arguments.
We let $\mu$ be the normalized pseudofinite counting measure on definable subsets of $K$.
That is, if $\phi(x_1,\ldots,x_n;y)$ is a formula in the language of fields and $a\in K^m$ is given by the sequence $(a_i \in K^n_i : i \in \Nn)$ then we have
\[
\mu \{b \in K : K \models \phi(a;b) \} = \lim_{i \to \upomega} \frac{|\{ b \in K_i : K_i \models \phi(a_i;b) \}| }{|K_i|}.
\]
Here we regard $\upomega$ as a point in the Stone-\u{C}ech compactification of $\Nn$ and take the limit $i \to \upomega$ accordingly.
The Lang-Weil estimates imply that if $A \subseteq K$ is definable then $\mu(A) = 0$ if and only if $A$ is finite; see \cite{CDM-finite}.
See \cite{anand-psueod-notes} for a detailed discussion of such measures in general and \cite{CDM-finite} for the construction of the measure on a general pseudofinite field.

\begin{proof}
Let $P$ be the set of non-zero squares in $K$ and $E$ be the set of $(a,b) \in \Aa^2(K)$ such that $a - b \in P$.
As above $E$ is $\cE_K$-open so it suffices to suppose that $U_0,U_1 \subseteq \Aa^1(K)$ are nonempty \'etale images and show that $E$ does not contain $U_0 \times U_1$.  
By Theorem~\ref{thm:3-strong}, $U_0$ and $U_1$ are infinite.
Fact~\ref{paley-ish}, the definition of $\mu$, and a routine ultraproduct argument together show that if $\beta_1,\ldots,\beta_k \in K$ are pairwise distinct then 
\[ \mu  \{ \alpha \in \Aa^1(K) : \alpha - \beta_1,\ldots,\alpha - \beta_k \in P \} \le 2^{1 - k}.\]
This yields
\[ \mu  \{ \alpha \in \Aa^1(K) : \forall \beta \in U_1 (\alpha - \beta \in P) \}  = 0, \]
so $U_0 \times U_1$ is not contained in $E$ as $\mu(U_0) > 0$.
\end{proof}

\section{Model-theoretic applications}

\subsection{A short  proof of  Theorem D from Theorems A and C}
\label{sec: tablefieldconj2}




\begin{proof}
Suppose $K$ is stable and not separably closed. By Theorem C(2) the \'etale open topology on $V(K)$  is Hausdorff for any quasi-projective $K$-variety $V$.
By Theorem A and Definition~\ref{sys-top:def}.\ref{st-1}, the map
\[ (V(K),\cE_V) \to (W(K),\cE_W)\]
is continuous for any morphism of $K$-varieties $f : V \to W$.  Specializing to the case where $V = W = \Aa^1$ and $f$ is an affine transformation, we see that the \'etale open topology on $\Aa^1(K) = K$ is Hausdorff and invariant under affine transformations.

\meno
As the topology is Hausdorff, we may take disjoint non-empty basic opens $U,U^* \subseteq K$.  The basic open sets in the \'etale-open topology are definable sets---they are images of morphisms of varieties.  Therefore $U,U^*$ are definable.
By Fact~\ref{fact:poizat} there is a unique additively and multiplicatively generic type in $K$.
As $U,U^*$ are disjoint they cannot both contain the generic type, so we suppose that $U$ is not generic.
By affine invariance we may suppose $0 \in U$.
Then $K \setminus U$ is multiplicatively generic in $K^\times$, hence there are $a_1, \ldots, a_n \in K^\times$ with
\[ K^\times = \bigcup_{i = 1}^n a_i \cdot (K \setminus U).\]
Equivalently $\{0\} = \bigcap_{i = 1}^n a_i \cdot U$.  By affine invariance of the topology, each $a_i \cdot U$ is open, and so $\{0\}$ is open.
By affine invariance again, every singleton is open and the topology is discrete.  By Theorem C.1, $K$ is not large.
\end{proof}

\subsection{Podewski's conjecture}
\label{section:podewski}
Recall that $K$ is \textbf{minimal} if every definable subset of $K$ is finite or co-finite.
\textbf{Podewski's conjecture} says that an infinite minimal field is algebraically closed.
Koenigsmann proved Podewski's conjecture for large fields~\cite{open-problems-ample}.
This follows from what is above.
Suppose that $K$ is large and minimal.
If $K$ is not separably closed we can produce $U,U^*$ as above and by largeness $U,U^*$ are both infinite, contradiction.
Finally, minimal fields are easily seen to be perfect, so $K$ is algebraically closed.

\meno
Wagner proved Podewski's conjecture in positive characteristic~\cite{wagner-minimal-fields}.
He shows that if $\Chara(K) \ne 0$ and $K$ is not algebraically closed then there is an infinite and co-infinite $\exists\forall$ definable subset of $K$.
Our proof produces two disjoint existentially definable infinite subsets of $K$ under the assumption that $K$ is large.
This is sharp as a quantifier free definable subset of $K$ is finite or cofinite.

\bibliographystyle{amsalpha}
\bibliography{ref}

\end{document}